\documentclass[12pt]{article}

\usepackage{preprint-layout}

\usepackage{hyperref}
\usepackage{bm}
\usepackage{units,setspace}
\usepackage{booktabs}

\newcommand{\addressumushort}{Mathematics and Mathematical Statistics, Ume{\aa}~University, Sweden}

\DeclareMathOperator{\diag}{diag}

\newcommand{\cproj}{P_0}

\newcommand{\pih}{\pi_h} 

\newcommand{\bfzero}{\boldsymbol{0}}

\newcommand{\bfx}{\boldsymbol{x}}
\newcommand{\bfe}{\boldsymbol{e}}
\newcommand{\bfu}{\boldsymbol{u}}
\newcommand{\bfv}{\boldsymbol{v}}
\newcommand{\bfw}{\boldsymbol{w}}
\newcommand{\bff}{\boldsymbol{f}}
\newcommand{\bfg}{\boldsymbol{g}}
\newcommand{\bfn}{\boldsymbol{n}}
\newcommand{\bfU}{\boldsymbol{U}}
\newcommand{\bfV}{\boldsymbol{V}}
\newcommand{\bfW}{\boldsymbol{W}}
\newcommand{\bfI}{\boldsymbol{I}}
\newcommand{\bfsig}{\boldsymbol{\sigma}}
\newcommand{\bfeps}{\boldsymbol{\epsilon}}
\newcommand{\bfrho}{\boldsymbol{\rho}}
\newcommand{\bfeta}{\boldsymbol{\eta}}
\newcommand{\bfPhi}{\boldsymbol{\Phi}}
\newcommand{\bfPsi}{\boldsymbol{\Psi}}

\usepackage{mathabx}
\newcommand{\tn}{\vvvert}

\newcommand{\V}{\bfV}
\newcommand{\Vh}{\bfV_{h}}
\newcommand{\Vhn}{\bfV_{h,n}}
\newcommand{\Vhnz}{P_0\bfV_{h,n}}

\newcommand{\mcK}{\mathcal{K}}
\newcommand{\mcL}{\mathcal{L}}
\newcommand{\mcA}{\mathcal{A}}
\newcommand{\mcB}{\mathcal{B}}
\newcommand{\mcF}{\mathcal{F}}

\title{Error Estimates for Finite Element Approximations of Viscoelastic Dynamics: The Generalized Maxwell Model}
\date{\today}
\author{Martin Bj\"orklund,\quad Karl Larsson,\quad Mats G. Larson}

\begin{document}

\maketitle

\begin{abstract}
We prove error estimates for a finite element approximation of viscoelastic dynamics based on continuous Galerkin in space and time, both in energy norm and in $L^2$ norm. The proof is based on an error representation formula using a discrete dual problem and a stability estimate involving the kinetic, elastic, and 
viscoelastic energies. To set up the dual error analysis and to prove the basic stability estimates, it is natural to formulate the problem as a first-order-in-time system involving evolution equations for the viscoelastic stress, the displacements, and the velocities. The equations for the viscoelastic stress can, however, be solved analytically in terms of the deviatoric strain velocity, and therefore, the viscoelastic stress can be eliminated from the system, resulting in a system for displacements and velocities.
\end{abstract}


\section{Introduction}
For many industrial engineering applications, particularly designing mechanical structures with polymeric materials, accurate dynamic simulation of time-dependent materials is necessary for reliable results.
Assuming small strains, the mechanical properties of time-dependent materials are well described by linear viscoelastic theory \cite{Tschoegl1997}, and there exist several models
for the constitutive equations of linear viscoelastic materials \cite{Findley1989}.
Based on various viscoelastic models formulated in the time or frequency domain, finite element methods incorporating viscoelastic dynamics have been developed, e.g., \cite{Golla1985, McTavish1993, Lesieutre1995, Lesieutre1996, Kaliske1997, Tillema2003, RivereShawWhiteman2007}.

In this paper, we derive a priori error estimates for a space-time finite element method for the dynamic simulation of viscoelastic materials modeled using the generalized Maxwell model, also known as the Wiechert model.
While there seem to exist quite a few papers regarding the a priori error analysis for \emph{quasi-static} viscoelastic problems, e.g., \cite{Shaw1994, Shaw1997, Shaw2000, Riviere2003, Rognes2010}, we have found only a few examples of a priori error estimates for finite element methods regarding \emph{dynamic} viscoelasticity for this model. Campo et al. \cite{Campo2005} proves an error estimate for a dynamic viscoelastic problem, but the analysis is limited to an explicit forward Euler-type scheme in time. A more rigorous analysis is performed by Rivi{\`e}re et al. \cite{RivereShawWhiteman2007} where a Crank--Nicolson type scheme in time and a nonsymmetric discontinuous Galerkin (dG) method in space are used. Recently, a more refined analysis was provided by Jang and Shaw \cite{MR4609819} where a symmetric dG formulation in space was considered. This method and analysis share several nice features with the present work, such as error bounds that are only linear with respect to the end time $T$, error bounds also in terms of the $L^2$-error, and representing the viscoelastic stresses compactly using internal vector-valued variables.

\paragraph{Contributions.} In this work, we present and analyze a finite element method for the dynamical simulation of linear viscoelastic materials modeled using the generalized Maxwell model. The method is based on approximation spaces that are continuous piecewise linear in time and continuous piecewise polynomial of order $p$ in space. In summary, our main contributions are:
\begin{itemize}
\item We prove space-time a priori error estimates in the end time energy norm and in the end time $L^2$ norm of the displacement field. The proofs of both estimates are based on the same error representation formula derived using a discrete dual problem, which provides a framework that can be used to analyze a wide range of methods.

\item By avoiding arguments involving Grönwall's lemma, our error estimates only linearly depend on the end time $T$, rather than exponentially as is the case when using Grönwall style arguments \cite{RivereShawWhiteman2007}. This property is shared with the analysis in \cite{MR4609819}.

\item The viscoelastic components of the stresses are, as in \cite{MR4609819}, compactly modeled using vector-valued internal variables. In addition, we demonstrate how these internal variables can be eliminated from the system of equations in each timestep. This leads to a greatly reduced system with only displacements and velocities as unknowns, where the internal variables at the next time are readily reconstructed using a simple formula.

\item In our analysis, the viscoelastic components of the stress are formulated in terms of a different differential operator than the elastic stress, the deviatoric strain, leading to a more involved analysis. In practice, viscoelastic stresses are often modeled using deviatoric strain \cite{Kaliske1997}.
\end{itemize}

\paragraph{Outline.}
The remainder of this paper is dispositioned as follows.
In Section~\ref{sect:visco_model} we present the governing equations of a
dynamic viscoelastic problem based on the generalized Maxwell linear viscoelastic model,
we derive a weak formulation of the problem, and we prove that the formulation satisfies a basic energy conservation law.
A finite element method based on piecewise linear interpolation in time and piecewise polynomial interpolation of order $p$ in space is derived in Section~\ref{sect:method}. An equivalent method where the viscoelastic variables for each Maxwell arm are eliminated is also derived, reducing the resulting system of equations to the same size as the purely elastic case. Both methods are shown to satisfy a discrete analogy to the basic energy conservation law.
In Section~\ref{sect:error_est}, we turn to proving an a priori error estimate in energy norm and $L^2$ norm. The main tool in this proof is a discrete dual problem which we use to derive a suitable error representation formula.
To support our theoretical results, we in Section~\ref{sect:numerical} study the end time convergence using a manufactured model problem in 3D. We also provide simulation results of a more realistic problem: modeling a radial shaft seal around a vibrating shaft. Finally, in Section~\ref{sec:conclusions}, we make some concluding remarks.

\section{The Generalized Maxwell Model}
\label{sect:visco_model}
\subsection{Governing Equations}
The displacement vector field $\bfu_0$ and stress tensor $\bm{\sigma}$ 
at time $t$ in a deformable material occupying a bounded domain $\Omega 
\subset \mathbb{R}^3$, with piecewise smooth boundary $\partial\Omega$ and outward pointing unit normal $\bfn$, satisfies
\begin{subequations}
\begin{alignat}{2}
\rho \ddot{\bfu}_0 - \bfsig \cdot \nabla &= \bff \qquad 
&& \text{in $\Omega \times [0,T]$}
\label{eq:equilibrium}
\\
\bfsig \bfn &= \bfg \qquad && \text{on $\Gamma_N \times [0,T]$}
\label{eq:neumann}
\\
\bfu_0 &= \bfzero \qquad && \text{on $\Gamma_D \times [0,T]$}
\label{eq:dirichlet}
\\
\bfu_0 &= \overline{\bfu}_0 \qquad && \text{in $\Omega$ at $t=0$}
\\
\dot{\bfu}_0 &= \overline{\bfv}_0 \qquad && \text{in $\Omega$ at $t=0$}
\end{alignat}
\end{subequations}
where $\bfsig \cdot \nabla$ is the row-wise divergence of $\bfsig$, $\rho$ is the material density, $\bm{f}(t)$ is a body force density, $\bm{g}(t)$ is a traction pressure, and $\overline{\bm{u}}_0$ and $\overline{\bm{v}}_0$ are initial displacements and velocities fulfilling \eqref{eq:dirichlet}. 
For the fixed and traction parts of the boundary we have $\Gamma_D \cup \Gamma_N = \partial\Omega$ and
$\Gamma_D \cap \Gamma_N = \emptyset$.
Furthermore, to completely describe the state at $t=0$ for a material which depends on the history, in this case a viscoelastic material, we either need the entire history of the material or its state representation in the specific material model.

\paragraph{Viscoelastic Model.}
Assuming an isotropic linear viscoelastic material based on the generalized Maxwell model, we separate the stress tensor into a sum of the elastic and viscoelastic components
\begin{equation}
\bfsig = \bfsig_{E}
+ \bfsig_{VE}
\end{equation}

The elastic part of the stress is governed by Hooke's law 
\begin{align}
\bfsig_{E} &= 2\mu \bfeps(\bfu_0) + \lambda \text{tr}(\bfeps(\bfu_0))\bfI
\end{align}
where $\mu$ and $\lambda$ are Lamé parameters, $\bfeps(\bfw) 
= \frac{1}{2}\left(\nabla\otimes\bfw +\bfw\otimes\nabla\right)$ 
is the linear strain tensor, and $\bfI$ is the $3\times 3$ 
identity matrix. We use the notation $\bfw\otimes\nabla = (\nabla\otimes\bfw)^T$.

The viscoelastic stress is given by the so-called Maxwell model
\begin{align}
\bfsig_{VE} &= \sum_{m=1}^M \bfsig_{VE}^m
\end{align}
where each contribution $\bfsig_{VE}^m$ is determined by the 
differential equation 
\begin{equation}
\dot{\bfsig}_{VE}^m + \frac{1}{\tau_m}\bfsig_{VE}^m = \kappa_m \bfe(\dot{\bfu}_0)
\end{equation}
with $\bfsig_{VE}^m(0)$ given,
and $\kappa_m$ and $\tau_m$ are material parameters denoted elastic modulus
and relaxation time, respectively. 
The deviatoric strain $\bm{e}$ 
is given by $\bfe( \bfw ) = \bfeps(\bfw) - \frac{1}{3} \text{tr}\left( \bfeps(\bfw) \right) \bfI$. Note that $\text{tr}(\bfe)=0$. 
Assuming $\bfsig_{VE}^m(0) = \bfzero$ and using Duhamel's formula 
we get the identity 
\begin{align}
\bfsig_{VE}^m(t)
&= \int_0^t \kappa_m \exp\left( - \frac{t-s}{\tau_m} \right) 
\bfe \left(\dot{\bfu}_0(s)\right) \, ds 
\label{eq:viscostress}
\end{align}
Thus $\bfsig_{VE}^m(t)$ is determined by the deviatoric strain velocity $\bfe(\dot{\bfu_0})$
 on the time interval $(0,t)$. Furthermore, 
we note that 
\begin{align}
\bfsig_{VE}^m(t) 
= \kappa_m \bfe\left( \int_0^t \exp\left( - \frac{t-s}{\tau_m} \right)  \dot{\bfu}_0(s) \, ds\right) 
\end{align}
since the integral in time commutes with the deviatoric strain 
given by a certain spatial derivative. Thus, introducing the notation 
\begin{equation}
\bfu^m_{VE} =  \int_0^t \exp \left( - \frac{t-s}{\tau_m} \right)  
\dot{\bfu}_0 (s) \, ds 
\label{eq:explicituve}
\end{equation}
we have the identity 
\begin{equation}
\bfsig_{VE}^m = \kappa_m  \bfe (\bfu_{VE}^m )
\end{equation}
and we note that $\bfu^m_{VE}$ satisfies the differential equation
\begin{equation}
\dot{\bfu}^m_{VE} + \frac{1}{\tau_m} \bfu^m_{VE} = \dot{\bfu}_0
\end{equation}
We denote $\bfu^m_{VE}$ viscoelastic velocities and remark that these variables will be used in conjunction with the displacements $\bfu_{0}$ as state variables for the stress $\bfsig = \bfsig(\bfu_{0},\bfu_{VE})$.
For compactness we here introduced the notation $\bfu_{VE}:=(\bfu_{VE}^1,\bfu_{VE}^2,\cdots,\bfu_{VE}^M)$, and we note that these variables are equivalent to the \emph{velocity form} of the internal variables used in \cite{MR4609819}.

\subsection{Weak Formulation in Space}

Introducing the notation $\bfu_1 = \dot{\bfu}_0$ 
we may write \eqref{eq:equilibrium} as a first-order-in-time system
\begin{subequations} \label{eq:strong}
\begin{alignat}{2}
\rho \dot{\bfu}_1 - \bfsig (\bfu_0,\bfu_{VE})\cdot \nabla  &= \bff &&
\label{eq:strong1}
\\
\dot{\bfu}_0 - \bfu_1 & = \bfzero &&
\label{eq:strong2}
\\
\dot{\bfu}_{VE}^m + \frac{1}{\tau_m} \bfu_{VE}^m - \bfu_1 &= \bfzero \ , &\qquad& m=1 \dots M
\label{eq:strong3}
\end{alignat}
\end{subequations}
where 
\begin{align}
\bfsig (\bfu_0,\bfu_{VE}) 
&= 
\bfsig_E (\bfu_0)
+
\bfsig_{VE} (\bfu_{VE})
\\&=
2\mu \bfeps(\bfu_0) + \lambda \text{tr}(\bfeps(\bfu_0))\bfI
+
\sum_{m=1}^M \kappa_m \bfe(\bfu_{VE}^m)
\label{eq:kdjsn}
\end{align}
Assuming an initial state for the viscoelastic state variables $\bfu_{VE}^m(0)$, $m=1\dots M$, such that $\bfu_{VE}^m(0)|_{\Gamma_D}=\bfzero$, \eqref{eq:strong2}--\eqref{eq:strong3} implies that $\bfu_1|_{\Gamma_D}=\bfu_{VE}^m|_{\Gamma_D}=\bfzero$.
To achieve a weak formulation with a suitable structure for our analysis, we replace \eqref{eq:strong2}--\eqref{eq:strong3} with the following two PDE
\begin{subequations}
\label{eq:strong2b}
\begin{alignat}{2}
-\bfsig_E(\dot{\bfu}_0 - \bfu_1)\cdot\nabla & = \bfzero && \qquad \text{in $\Omega \times [0,T]$}
\\
\bfsig_E(\dot{\bfu}_0 - \bfu_1)\bfn & = \bfzero && \qquad \text{on $\Gamma_N \times [0,T]$}
\\
\dot{\bfu}_0 - \bfu_1 & = \bfzero && \qquad \text{on $\Gamma_D \times [0,T]$}
\label{eq:strong2bc}
\end{alignat}
\end{subequations}
and
\begin{subequations}
\label{eq:strong3b}
\begin{alignat}{2}
-\bfsig_{VE}^m\bigl(
\dot{\bfu}_{VE}^m + \frac{1}{\tau_m} \bfu_{VE}^m - \bfu_1
\bigr) \cdot \nabla &= \bfzero  &&\qquad \text{in $\Omega \times [0,T]$}
\\
\bfsig_{VE}^m\bigl(
\dot{\bfu}_{VE}^m + \frac{1}{\tau_m} \bfu_{VE}^m - \bfu_1
\bigr) \bfn &= \bfzero  &&\qquad \text{on $\Gamma_N \times [0,T]$}
\\
\dot{\bfu}_{VE}^m + \frac{1}{\tau_m} \bfu_{VE}^m - \bfu_1 &= \bfzero && \qquad \text{on $\Gamma_D \times [0,T]$}
\label{eq:strong3bc}
\end{alignat}
\end{subequations}
for $m=1\dots M$. Note that for $\bfu_0|_{\Gamma_D} = \bfu_1|_{\Gamma_D} = \bfu_{VE}^m|_{\Gamma_D} = \bfzero$ the ODE boundary conditions \eqref{eq:strong2bc} and \eqref{eq:strong3bc} are always fulfilled. 
Assuming $\Gamma_D$ has a non-zero area measure, this is equivalent to the original formulation \eqref{eq:strong2}--\eqref{eq:strong3}.

Let $\V := \{ \bfv \in [H^1(\Omega)]^3: \bfv|_{\Gamma_D}=\bfzero \}$
and let $\left( a , b \right) := \int_\Omega a \cdot b \, dx$ where $\cdot$ denotes the natural inner product for elements $a,b$.
Multiplying \eqref{eq:strong1} by $\bfv_1 \in \V$, \eqref{eq:strong2b} by $\bfv_0 \in \V$,
and \eqref{eq:strong3b} by $\bfv_{VE}^m \in \V$,
integrating over $\Omega$, and applying Green's formula results in the system
\begin{subequations}
\begin{align}
(\rho \dot{\bfu}_1,\bfv_1) + (\bfsig ,\bfv_1 \otimes \nabla)
&= (\bff, \bfv_1) + (\bfg, \bfv_1)_{\Gamma_N}
\\
(\bfsig_E(\dot{\bfu}_0 - \bfu_1), \bfv_0 \otimes\nabla)
&=\bfzero
\label{eq:weak2a}
\\
(\bfsig_{VE}^m(
\dot{\bfu}_{VE}^m + \frac{1}{\tau_m} \bfu_{VE}^m - \bfu_1
), \bfv_{VE}^m \otimes \nabla )
&=\bfzero
\label{eq:weak3a}
\end{align}
\end{subequations}
for $m=1\dots M$.
Observing that 
\begin{align}
(\bfsig,\bfv_1 \otimes \nabla)
&=(\bfsig_E(\bfu_0),\bfv_1 \otimes \nabla) + \sum_{m=1}^M(\bfsig_{VE}^m(\bfu_{VE}^m),\bfv_1 \otimes \nabla)
\\
&=(\bfsig_E(\bfu_0),\bfeps(\bfv_1)) + \sum_{m=1}^M(\bfsig_{VE}^m(\bfu_{VE}^m),\bfe(\bfv_1)) \,,
\end{align}
performing the corresponding calculation for \eqref{eq:weak2a}--\eqref{eq:weak3a}, and introducing the forms 
\begin{equation}
a_{E}(\bfv,\bfw) :=  (\bfsig_E(\bfv), \bfeps(\bfw)) \,,
\qquad 
a_{VE}^m(\bfv,\bfw) :=  (\bfsig_{VE}^m(\bfv), \bfe(\bfw))
\end{equation}
we thus arrive at the problem: find $(\bfu_1,\bfu_0,\underbrace{\bfu_{VE}}_{\mathclap{:= (\bfu_{VE}^1,\dots,\bfu_{VE}^M)}}) \in [\V]^{2+M}$
such that
\begin{subequations}
\begin{alignat}{2}
(\rho \dot{\bfu}_1,\bfv_1) + a_E(\bfu_0,\bfv_1) 
+ \sum_{m=1}^M a^m_{VE}(\bfu_{VE}^m,\bfv_1) 
&= (\bff, \bfv_1) + (\bfg, \bfv_1)_{\Gamma_N} \, , \quad&&\forall\bfv_1\in\V
\label{eq:jnfdv}
\\
a_E(\dot{\bfu}_0,\bfv_0) - a_E(\bfu_1,\bfv_0) &=0 \, , \quad&&\forall\bfv_0\in\V
\label{eq:bofid}
\\
\sum_{m=1}^M a_{VE}^m(\dot{\bfu}_{VE}^m + \frac{1}{\tau_m} \bfu_{VE}^m 
- \bfu_1,\bfv_{VE}^m) &= 0
 \, , \quad&&\forall\bfv_{VE}^m\in\V
\label{eq:bkdfu}
\end{alignat}
\end{subequations}
for $m=1,2,\dots,M$.
Collecting the equations, we get the following variational problem.
\paragraph{Weak Formulation.}
Find 
$\bfu=(\bfu_1, \bfu_0, \bfu_{VE}) \in [\V]^{2+M}$ such that 
\begin{align}
\boxed{
A(\dot{\bfu},\bfv)
+ B( \bfu, \bfv )
= L(\bfv)
\,, \qquad \forall \bfv = (\bfv_1,\bfv_0,\bfv_{VE}) \in [\V]^{2+M}
}
\label{eq:weakform}
\end{align}
where the forms are given by
\begin{align}
A(\dot{\bfu},\bfv)
&:=(\rho \dot{\bfu}_1,\bfv_1) + a_E(\dot{\bfu}_0,\bfv_0) 
+  \sum_{m=1}^M a_{VE}^m( \dot{\bfu}_{VE}^m,\bfv^m_{VE})
\\[0.8em]
B( \bfu , \bfv)
&:=a_E(\bfu_0,\bfv_1) - a_E(\bfu_1,\bfv_0)
\\ \nonumber
&\qquad +\sum_{m=1}^M a_{VE}^m(\bfu_{VE}^m,\bfv_1) 
- a_{VE}^m(\bfu_1,\bfv^m_{VE})
\\ \nonumber
&\qquad +\sum_{m=1}^M \frac{1}{\tau_m}a_{VE}^m( \bfu_{VE}^m,\bfv^m_{VE})
\\[0.8em]
L(\bfv) &:= (\bff,\bfv_1) + (\bfg,\bfv_1)_{\Gamma_N}
\end{align}
Note that $A(\cdot,\cdot)$ is a positive definite symmetric form and that we can write $B(\cdot,\cdot)$ as a sum of a positive semi-definite symmetric form and a 
skew-symmetric form
\begin{equation}
B(\cdot,\cdot) = B_{Sym}(\cdot,\cdot) + B_{Skew}(\cdot,\cdot)
\end{equation}
where
\begin{align}
B_{Sym}(\bfu,\bfv) &:= \sum_{m=1}^M \frac{1}{\tau_m}a_{VE}^m( \bfu_{VE}^m,\bfv^m_{VE})
\\
B_{Skew}(\bfu,\bfv) &:=
a_E(\bfu_0,\bfv_1) - a_E(\bfu_1,\bfv_0)
 +\sum_{m=1}^M a_{VE}^m(\bfu_{VE}^m,\bfv_1) 
- a_{VE}^m(\bfu_1,\bfv^m_{VE})
\end{align}

\subsection{Norms and Basic Conservation Law}

We first introduce the elastic and viscoelastic energy norms on $\V$ given by
\begin{align}
&\tn \bfv \tn_E^2 := a_E(\bfv,\bfv)
\qquad \text{and} \qquad 
\tn \bfv \tn_{VE,m}^2 := a_{VE}^m(\bfv,\bfv)
\label{eq:defenergyEVE}
\end{align}
respectively, and note that assuming $\Gamma_D\neq\emptyset$ both these norms are equivalent to the standard $[H^1(\Omega)]^3$ norm through Korn's inequality.
The energy norm on $[\V]^{2+M}$ for the system is defined
\begin{align}
\tn (\bfv_1, \bfv_0, \bfv_{VE}) \tn_A^2 &:=
A((\bfv_1,\bfv_0,\bfv_{VE}),(\bfv_1,\bfv_0,\bfv_{VE}))
\\&\hphantom{:}=
\rho\|\bfv_1\|^2_{L^2(\Omega)} + \tn \bfv_0 \tn_E^2 +  
\sum_{m=1}^M \tn \bfv_{VE}^m\tn^2_{VE,m} \label{eq:energy-norm-terms}
\end{align}
where $\| \cdot \|_{L^2(\Omega)}$ is the usual $L^2(\Omega)$ norm.
The three terms in \eqref{eq:energy-norm-terms} have physical interpretations as potential energy, elastic energy and viscoelastic energy, respectively.
In the time domain, we will also use the $L^\infty(I)$ norm, i.e., the max norm in time over a time interval $I$.
Let $\| \cdot \|_{L^\infty(I,E)}$ denote the max norm in time of $\tn \cdot \tn_E$, i.e.,
\begin{align}
\| \bfu_0 \|_{L^\infty(I,E)} := \sup_{t\in I} \tn \bfu_0(t) \tn_E
\end{align}
and
analogously define $\| \cdot \|_{L^\infty(I,[VE,m])}$, $\| \cdot \|_{L^\infty(I,H^2(\Omega))}$, etc.
On the complete simulation time interval $I=[0,T]$, we use abbreviated notations $\| \cdot \|_{L^\infty(E)}:=\| \cdot \|_{L^\infty([0,T],E)}$, etc.

\begin{lem}[Energy Conservation] \label{lemma:basicconservationlaw}
Let $\bfu(t)=(\bfu_1, \bfu_0, \bfu_{VE})(t) \in [\V]^{2+M}$ be the time dependent solution to the variational problem \eqref{eq:weakform} for all times $0 \leq t \leq T$.
For $\{t_0,t_1\} \in [0,T]$ and $\bff=\bfg=\bfzero$ the following conservation law holds
\begin{align}
\tn \bfu(t_1) \tn_A^2
+
\sum_{m=1}^M \frac{2}{\tau_m} \int_{t_0}^{t_1} \tn \bfu_{VE}^m(s) \tn_{VE,m}^2 \, ds
&=
\tn \bfu(t_0) \tn_A^2
\label{eq:basicconservationlaw}
\end{align}
\end{lem}
\begin{proof}
In \eqref{eq:weakform} setting 
$\bfv = \bfu$ and integrating over the time interval $(t_0,t_1)$ yields the conservation law.
\end{proof}
\begin{rem}
Note that Lemma~\ref{lemma:basicconservationlaw} is a generalization of the standard 
energy conservation law for elastic materials, i.e.,
\begin{align}
&\rho\|\bfu_1(t_1) \|^2_{L^2(\Omega)} + \tn \bfu_0(t_1)\tn_E^2
=
\rho\|\bfu_1(t_0) \|^2_{L^2(\Omega)} + \tn \bfu_0(t_0)\tn_E^2 
\end{align}
The additional terms in \eqref{eq:basicconservationlaw} consist of the viscoelastic energy
$\sum_{m=1}^M \tn \bfu_{VE}^m \tn_{VE,m}^2$,
hidden in the energy norm $\tn\cdot\tn_A$,
and the dissipated work
$
\sum_{m=1}^M \frac{2}{\tau_m} \int_{t_0}^{t_1} \tn \bfu_{VE}^m(s) \tn_{VE,m}^2 \, ds
$.
\end{rem}

\section{The Finite Element Method}
\label{sect:method}
\subsection{Finite Element Spaces}
Let $\mcK_h$, with $0< h \leq h_0$, be a family of quasiuniform partitions, with 
meshparameter $h$, of $\Omega$ into shape regular tetrahedra $K$ and let each spatial component in
$\Vh\subset\V$ be the space of piecewise continuous polynomials of order $p$ defined on $\mcK_h$.

Next let $0=t_0<t_1<\dots<t_N=T$ be a partition of $(0,T]$ into time intervals $I_n=(t_{n-1},t_n]$ of
length $k_n = t_n - t_{n-1}$ and we define the following two finite element spaces on each space-time slab $I_n \times \Omega$
\begin{alignat}{2}
\Vhn &:= Q_1(I_n) \times \Vh
\qquad\qquad&
\Vhnz &:= Q_0(I_n) \times \Vh
\end{alignat}
where $Q_\ell(I_n)$ is the space of polynomials of degree less or equal 
to $\ell$.

\begin{rem}
The main reason for using only linear approximation in time while allowing for higher-order approximation in space is to simplify the analysis. By limiting the analysis to piecewise linears, we can readily bound expressions formulated on a time interval in terms of the corresponding expressions at the nodal points in time.

In practical applications, a benefit of using a higher-order approximation in space is that the deteriorative effects of volumetric locking are diminished. This is of great relevance when simulating viscoelastic materials since elastomers are usually incompressible.
\end{rem}

\subsection{Interpolation}
We here define interpolation operators in space and in time.
In our error estimates, we use the notation
$a \lesssim b$ to describe inequalities of the form $a \leq C b$ with a constant $C$ independent of the mesh size $h$ and the timestep $k = \max_{n=1,\dots,N} k_n$.

\paragraph{Interpolation in Space.}

We let $\pi_h:[\V]^{2+M} \rightarrow [\Vh]^{2+M}$ denote an interpolant in space defined as $\pi_h:=\diag(R_{E},R_{E},R_{VE},\dots,R_{VE}) $ where $R_{E}$ and $R_{VE}$ are Ritz projections defined
\begin{alignat}{2}
a_E(R_{E} \bfv - \bfv,\bfw) &= 0 \, , \qquad &&\forall\bfw\in\Vh
\\
(\bfe(R_{VE} \bfv - \bfv),\bfe(\bfw)) &= 0 \, , \qquad &&\forall\bfw\in\Vh
\end{alignat}
The latter definition implies $a_{VE}^m(R_{VE} \bfv - \bfv,\bfw)=0$ for all $\bfw\in\Vh$ and $m=1\dots M$.

\begin{lem}[Energy Norm Interpolation Error] \label{lemma:interpolation}
For $\bfv = (\bfv_1,\bfv_0,\bfv_{VE})$, with each vector in $[H^r(\Omega)]^3$, $r\geq 1$, and
$\pi_h \bfv = (R_E \bfv_1,R_E \bfv_0, R_{VE} \bfv_{VE}^1, \cdots, R_{VE} \bfv_{VE}^M)$ there
exists a constant $h_0 > 0$ such that for mesh sizes $0 < h < h_0$ it holds
\begin{align}
\tn \bfv - \pi_h \bfv \tn_A
&\lesssim
h^{s} \left\| \bfv_1 \right\|_{H^s(\Omega)}
+
h^{s-1} \left(
\left\| \bfv_0 \right\|_{H^s(\Omega)} + \sum_{m=1}^M \left\| \bfv_{VE}^m \right\|_{H^s(\Omega)}
\right)
\end{align}
where $s=\min(r,p+1)$ and the constant in $\lesssim$ is independent of $h$ but typically dependent of $\Omega$.
\end{lem}
\begin{proof}
Due to the construction of the interpolant via Ritz projections, the interpolation estimate directly follows from standard finite element error estimates for the static problem, see, e.g., \cite{BrennerScott2008}.
\end{proof}

\paragraph{Interpolation in Time.}
We let $\pi_k$ denote the Lagrange interpolant onto continuous
piecewise linear  functions in the time 
direction with nodes at $\{t_n\}_{n=0}^N$.
The combination of both the interpolants in space and time we denote by $\pi_{hk}$ and note that $\pi_k$ and $\pih$ commutes, i.e., that
\begin{align}
\pi_{hk} \bfv := \pi_k \pih \bfv = \pih \pi_k \bfv
\end{align}

We will also need the $L^2$-projection $\cproj v$ of a function $v$ 
onto piecewise constants in time, defined by
\begin{equation}
\int_{I_n} (\cproj v-v) w \, dt
= 0 \, , \qquad 
\forall w \in Q_0(I_n) \, , \quad n=1,2,\dots, N
\label{eq:P0}
\end{equation}
and make use of the standard max norm estimates
\begin{align}
\| v - \pi_k v \|_{L^\infty([0,T])} &\leq \frac{1}{8} k^2 \| \ddot{v} \|_{L^\infty([0,T])}
\label{eq:interp_pik}
\\
\| v - \cproj v \|_{L^\infty([0,T])} &\lesssim k \| \dot{v} \|_{L^\infty([0,T])}
\end{align}
see, e.g., \cite{Eriksson1996}.

\subsection{The Method}
The finite element method takes the form: given $\bfU(t_0)$ 
find $\bfU = (\bfU_1,\bfU_0, \bfU_{VE}) \in [\Vhn]^{2+M}$ such that 
$\bfU^+(t_{n-1}) = \bfU^-(t_{n-1})$, where $\bfU^\pm$ denotes right and left-hand limits in time, and
\begin{equation}
\int_{I_n} A(\dot{\bfU},\bfW) +  B(\bfU,\bfW) \, dt
= \int_{I_n} L_k(\bfW) \, dt
\, , \qquad \forall \bfW \in [\Vhnz]^{2+M}
\label{eq:discreteprob}
\end{equation}
for $n=1,\dots,N$. Here, an approximate linear functional $L_k$ defined
\begin{align}
L_k((\bfv_1,\bfv_0,\bfv_{VE})) &:= (\bff,\bfv_1) + (\bfg_k,\bfv_1)_{\Gamma_N}
\end{align}
is used where $\bfg_k$ is a piecewise linear approximation to $\bfg$ in time, i.e., $\bfg_k \in Q_1(I_n)\times L^2(\Gamma_N)$.
The analysis below shows that $\bfg_k := \pi_k \bfg$ is a suitable choice.
As $\Vh \subset \V$, by integrating \eqref{eq:weakform} and subtracting
\eqref{eq:discreteprob}, we readily get that the following approximate Galerkin orthogonality holds
which accounts for boundary data approximation
\begin{align}
\int_{I_n} A(\dot{\bfu}-\dot{\bfU},\bfW) +  B(\bfu - \bfU,\bfW) \, dt
=
(\bfg - \bfg_k, \bfW_1)_{\Gamma_N}
\, , \quad \forall \bfW \in [\Vhnz]^{2+M}
\label{eq:go}
\end{align}

Introducing discrete differential operators $\mcL_{h,E},\mcL_{h,VE}^m:\Vh\to\Vh$ such 
that 
\begin{alignat}{2}
\label{eq:discrete-operator-E}
(\mcL_{h,E} \bfU_0, \bfv) &= a_E(\bfU_0,\bfv) \, , \qquad &&\forall\bfv\in\Vh
\\
\label{eq:discrete-operator-VE}
 (\mcL_{h,VE}^m \bfU_{VE}^m, \bfv) &= a_{VE}^m(\bfU_{VE}^m,\bfv) \, , \qquad &&\forall\bfv\in\Vh \, , \quad m=1,\dots,M
\end{alignat}
and discrete matrix operators $\mcA_h$ and $\mcB_h$ such that
\begin{align}
\mcA_h \bfU :=
\begin{pmatrix}
\rho I & 0  &  0  &  \cdots & 0
\\
 0 & \mcL_{h,E} & 0 & \hdots & 0
\\
 0 & 0 & \mcL_{h,VE}^1 & \ddots &  \vdots
\\
\vdots & \vdots & \ddots & \ddots & 0
\\
 0 & 0 & \cdots & 0 & \mcL_{h,VE}^M
\end{pmatrix}
\begin{pmatrix}
\bfU_1
\\
\bfU_0
\\
\bfU_{VE}^1
\\
\vdots
\\
\bfU_{VE}^M
\end{pmatrix}
\end{align}
and
\begin{align}
\mcB_h \bfU :=
\begin{pmatrix}
0 & \mcL_{h,E} & \mcL_{h,VE}^1  & \cdots & \mcL_{h,VE}^M
\\
-\mcL_{h,E} & 0 & 0 & \cdots & 0
\\
-\mcL_{h,VE}^1 & 0 & \frac{1}{\tau_1}\mcL_{h,VE}^1 & \ddots &  \vdots
\\
\vdots  & \vdots  & \ddots & \ddots & 0
\\
-\mcL_{h,VE}^M & 0 & \cdots & 0 & \frac{1}{\tau_M}\mcL_{h,VE}^M
\end{pmatrix}
\begin{pmatrix}
\bfU_1
\\
\bfU_0
\\
\bfU_{VE}^1
\\
\vdots
\\
\bfU_{VE}^M
\end{pmatrix}
\end{align}
we may write \eqref{eq:discreteprob} in terms of $L^2$ inner products as
\begin{equation}
\int_{I_n} (\mcA_h\dot{\bfU},\bfW) +  (\mcB_h\bfU,\bfW) \, dt
= \int_{I_n} L_k(\bfW) \, dt
\label{eq:discreteprobsystem}
\end{equation}
This notation in terms of discrete differential operators will be convenient in the proof of the $L^2$ error estimate.

\subsection{Elimination of Viscoelastic Variables} \label{sec:elimination}
Looking at \eqref{eq:discreteprob} we note that it is possible to explicitly express $\bfU_{VE}$ in terms
of $(\bfU_1,\bfU_0)$ by solving $M$ first order ODE.
We use this to formulate a reduced method with only $\bfU_r = (\bfU_1,\bfU_0)$ as unknowns.

As we in the discrete problem \eqref{eq:discreteprob} for the third equation use test functions in $\Vhnz$
we may express the discrete viscoelastic velocities $\bfU_{VE}^m$ on $I_n$ via the ODEs
\begin{align}
\cproj \left(
\dot{\bfU}_{VE}^m - \bfU_1 + \frac{1}{\tau_m} \bfU_{VE}^m
\right) = 0 \, ,
\qquad\quad m=1 \dots M
\end{align}
For $\bfU_1, \bfU_{VE}^m \in \Vhn$ these ODEs can be solved explicitly
yielding the following incremental update formulas
\begin{align}
\bfU_{VE}^m(t_n) =
\alpha_{k_n}^m
\left( \bfU_1(t_n) + \bfU_1(t_{n-1}) \right)
+
\beta_{k_n}^m
\bfU_{VE}^m(t_{n-1}) \,,
\qquad\quad m=1 \dots M
\label{eq:discodesol}
\end{align}
where
$\alpha_k^m := \frac{k}{2 + \nicefrac{k}{\tau_m}}$
and
$\beta_k^m := \frac{2 - \nicefrac{k}{\tau_m}}{{2 + \nicefrac{k}{\tau_m}}}$.
Setting $\bfU^m_{VE}(t_0)=\bfzero$ we obtain a method 
involving only the discrete velocities $\bfU_1$ and discrete displacements $\bfU_0$.

The last term in \eqref{eq:discodesol} can in each timestep be moved to the 
right hand side and we obtain the reduced method: find 
$\bfU_r = (\bfU_1,\bfU_0) \in \Vhn \times \Vhn$ 
such that 
\begin{equation}
\int_{I_n} A_r(\dot{\bfU}_r,\bfW_r) +  B_r(\bfU_r,\bfW_r) \, dt = \int_{I_n} L_{k,r}(\bfW_r) \, dt \,, \qquad \forall \bfW_r \in \Vhnz \times \Vhnz
\label{eq:reducedmeth}
\end{equation}
where
\begin{align}
A_r((\dot{\bfu}_1,\dot{\bfu}_0),(\bfv_1,\bfv_0))
&:=(\rho \dot{\bfu}_1,\bfv_1) + a_E(\dot{\bfu}_0,\bfv_0) 
\\
B_r( ({\bfu}_1,{\bfu}_0),(\bfv_1,\bfv_0)) 
&:=a_E(\bfu_0,\bfv_1) - a_E(\bfu_1,\bfv_0)
+\sum_{m=1}^M a_{VE}^m(\alpha_k^m \bfu_1,\bfv_1) 
\\ \label{eq:Lr}
L_{k,r}((\bfv_1,\bfv_0) &:= (\bff,\bfv_1) + (\bfg_k,\bfv_1)_{\Gamma_N}
\\ &\qquad \nonumber
- \sum_{m=1}^M a_{VE}^m\left( \left(\frac{1 + \beta^m_k}{2}\right) \bfU_{VE}^m(t_{n-1}),\bfv_1 \right)
\end{align}
Note that $\bfU_{VE}^m(t_{n-1})$ in \eqref{eq:Lr} is already known via \eqref{eq:discodesol}.

\paragraph{Algorithm for Reduced Method.}
Let the initial state $\bfU(t_0) = (\bfU_1(t_0),\bfU_0(t_0),\bfU_{VE}(t_0))$ be given.
Starting with $n=1$, compute the next state $\bfU(t_n)$ using the steps:
\begin{itemize}
\item Solve for $(\bfU_1(t_n),\bfU_0(t_n))$ using \eqref{eq:reducedmeth}.
\item Reconstruct $\bfU_{VE}^m(t_n)$, $m=1 \dots M$, using \eqref{eq:discodesol}.
\end{itemize}

\begin{rem}
Note that the reconstructed solution based on the reduced method \eqref{eq:reducedmeth}
is identical to the solution from the complete discrete method \eqref{eq:discreteprob}. Thus all results below for the complete discrete method are also valid for the reduced method.
\end{rem}

\subsection{Discrete Conservation Law}

\begin{lem}[Conservation Law] \label{lemma:disccons}
Let $\bfU = (\bfU_1,\bfU_0,\bfU_{VE})$ be the solution to the discrete problem \eqref{eq:discreteprob}
with $\bff = \bfg_k = \bfzero$.
The following conservation law is then satisfied
\begin{align}
\boxed{
\tn \bfU(t_1) \tn^2_A
+  \sum_{m=1}^M \frac{2}{\tau_m} \int_{t_0}^{t_1} \tn \cproj\bfU_{VE}^m \tn^2_{VE,m} \, dt
= \tn \bfU(t_0) \tn^2_A
}
\label{eq:discretecons}
\end{align}
where $t_0$ and $t_1$ are any two nodes in the time discretization.
\end{lem}
\begin{proof}
Setting $\bfW = \cproj\bfU$ in the method \eqref{eq:discreteprob} we obtain
\begin{align}
0 &= \int_{I_n} A(\dot{\bfU},\cproj \bfU) + B(\bfU, \cproj \bfU) \, dt
\\
&= \int_{I_n} A(\dot{\bfU}, \bfU) + B( \cproj \bfU, \cproj \bfU) \, dt
\\
&= \int_{I_n} \frac{1}{2} \partial_t A(\bfU, \bfU) + B_{Sym}(\cproj \bfU, \cproj \bfU) \, dt
\end{align}
Summing all contributions from node $t_0$ to node $t_1$ completes the proof.
\end{proof}

\section{A Priori Error Estimates}
\label{sect:error_est}
\subsection{Error Representation Formula}

Consider the following discrete dual problem:
given $\bfPhi(T) = \bfPsi$
find $\bfPhi \in [\Vhn]^{2+M}$ such 
that $\bfPhi^+(t_{n})=\bfPhi^-(t_{n})$ and
\begin{equation}
\int_{I_n} -A(\bfW,\dot{\bfPhi}) + B(\bfW,\bfPhi) \, dt = 0 \,,
\qquad \forall \bfW \in [\Vhnz]^{2+M}
\label{eq:discdualprob}
\end{equation}
We split the error into two components
\begin{equation}
\bfu - \bfU
= \underbrace{\bfu - \pi_{hk} \bfu}_{\bfrho}
+ 
\underbrace{\pi_{hk} \bfu - \bfU}_{\bfeta} = 
\bfrho + \bfeta
\label{eq:errsplit}
\end{equation}
Setting $\bfW = \cproj \bfeta  \in [\Vhnz]^{2+M}$ and summing over all intervals $I_n$ we 
get 
\begin{align}
0 &= \int_0^T -A(\cproj\bfeta,\dot{\bfPhi}) + B(\cproj\bfeta,\bfPhi)   \, dt
\\
&= \int_0^T -A(\bfeta,\dot{\bfPhi}) + B(\bfeta,\cproj \bfPhi)  \, dt
\\
&= A(\bfeta(0),\bfPhi(0)) 
- A(\bfeta(T),\bfPhi(T))
+ \int_0^T A(\dot{\bfeta},\bfPhi) +  B(\bfeta,\cproj \bfPhi) \, dt
\label{eq:kdfgjc}
\\
&= A(\bfeta(0),\bfPhi(0)) 
- A(\bfeta(T),\bfPhi(T))
+ \int_0^T A(\dot{\bfeta},\cproj \bfPhi) +  B(\bfeta,\cproj \bfPhi) \, dt
\\
&= A(\bfeta(0),\bfPhi(0)) 
- A(\bfeta(T),\bfPhi(T))
\label{eq:kdfgjd}
\\&\quad
- \int_0^T A(\dot{\bfrho},\cproj \bfPhi) +  B(\bfrho,\cproj \bfPhi) - (\bfg - \bfg_k, P_0\bfPhi_1)_{\Gamma_N} \, dt
\nonumber
\\
&= A(\bfeta(0),\bfPhi(0)) 
- A(\bfeta(T),\bfPsi)
\\&\quad
- \int_0^T A(\dot{\bfrho},\cproj \bfPhi) +  B(\bfrho,\cproj \bfPhi) - (\bfg - \bfg_k, P_0\bfPhi_1)_{\Gamma_N} \, dt
\nonumber
\end{align}
where we in \eqref{eq:kdfgjd} use the splitting of the error \eqref{eq:errsplit} and the approximate Galerkin orthogonality of the primal problem \eqref{eq:go}.
Thus we arrive at the error representation formula
\begin{equation}
\boxed{
A(\bfeta(T),\bfPsi) =  A(\bfeta(0),\bfPhi(0))
- \int_0^T A(\dot{\bfrho},\cproj \bfPhi) +  B(\bfrho,\cproj \bfPhi) - (\bfg - \bfg_k, P_0\bfPhi_1)_{\Gamma_N} \, dt
}
\label{eq:errorrep}
\end{equation}
which will serve as a basis for proving our a priori error estimates below.


\subsection{Stability Estimate for the Discrete Dual Problem}

\begin{lem}[Properties of the Discrete Dual Problem] The solution to the discrete dual problem \eqref{eq:discdualprob} satisfies 
the following estimates:
\begin{enumerate}
\item Conservation law
\begin{align}
\tn \bfPhi(t_{i}) \tn^2_A
+  \sum_{m=1}^M \frac{2}{\tau_m} \int_{t_{i}}^{t_j} \tn \cproj\bfPhi_{VE}^m\tn^2_{VE,m} \, dt
&=
\tn \bfPhi(t_j) \tn^2_A
\label{eq:dualdiscretecons}
\end{align}

\item Stability estimate required for energy error estimate 
\begin{align}
\| \dot{\bfPhi}_1 \|_{L^\infty(H^{-1}(\Omega))}^2
+
\| \bfPhi \|_{L^\infty(A)}^2
+
\sum_{m=1}^M \frac{2}{\tau_m}
\int_0^T \tn \cproj\bfPhi_{VE}^m \tn_{VE,m}^2 \, dt
\lesssim
\tn \bfPsi \tn^2_{A}
\label{eq:dualstabest}
\end{align}

\item Stability estimate required for $L^2$ error estimate ($\bfPsi_1=\bfPsi_{VE}^m=\bfzero$)
\begin{align}
&
\| \dot{\bfPhi}_1 \|_{L^\infty(L^2(\Omega))}^2
+
\bigl\|\mcL_{h,E} \bfPhi_0 + \sum_{m=1}^M \mcL_{h,VE}^m \bfPhi_{VE}^m \bigr\|_{L^\infty(L^2(\Omega))}^2
+
\| \bfPhi_1 \|_{L^\infty(E)}^2
\nonumber \\ & \qquad\qquad\qquad\quad
+ \sum_{m=1}^M
\| -\bfPhi_1 + \tau_m^{-1 }\bfPhi_{VE}^m \|_{L^\infty([VE,m])}^2
\lesssim
\| \mcL_{h,E} \bfPsi_0 \|_{L^2(\Omega)}^2
\label{eq:L2dualstabest}
\end{align}
\end{enumerate}
\end{lem}

\begin{rem}
In case the viscoelastic stresses would be formulated using the same differential operator as the elastic stress, stability of the terms involving $\bfPhi_{VE}^m$ would not be required for proving our error estimates.
\end{rem}

\begin{proof}
\textbf{(i) Conservation law.\ \ }
Choosing $\bfW=\cproj \bfPhi$ in the discrete dual problem \eqref{eq:discdualprob}
and summing over all intervals $\{I_n\}_{i+1}^j$ we get
\begin{align}
0 &= \int_{t_i}^{t_j} -A(\cproj\bfPhi,\dot{\bfPhi}) + B(\cproj\bfPhi,\bfPhi)   \, dt
\\
&= \int_{t_i}^{t_j} -A(\bfPhi,\dot{\bfPhi}) + B(\cproj\bfPhi , \cproj \bfPhi)  \, dt
\\
&= \int_{t_i}^{t_j} -\frac{1}{2}\partial_t A(\bfPhi,\bfPhi) + B_{Sym}(\cproj\bfPhi , \cproj \bfPhi)  \, dt
\end{align}
and the conservation law follows.

\paragraph{(ii) Stability for energy estimate.}
Choosing $t_j=T$ in the conservation law directly yields inequalities
\begin{align}
\sum_{m=1}^M \frac{2}{\tau_m}
\int_{t_i}^T \tn \cproj\bfPhi_{VE}^m \tn_{VE,m}^2 \, dt \leq  \tn \bfPhi(T) \tn_A^2 = \tn \bfPsi \tn_A^2
\end{align}
and
\begin{align} \label{eq:sldkgn}
\tn \bfPhi(t_{i}) \tn_A \leq \tn \bfPhi(T) \tn_A = \tn \bfPsi \tn_A \,, \qquad
\forall i\in\{0,1,\dots,N\}
\end{align}
where the latter stability is valid at the nodal points in time. Utilizing linear interpolation in time, with a nodal basis $\{\phi_i(t)\}_{i=0}^N$, we, by the triangle inequality, can bound the energy norm within a time-interval $I_n$ by its values at the nodes
\begin{align} \label{eq:time-max-start}
\sup_{t \in I_n}
\tn \bfPhi(t) \tn_{A} &= \sup_{t \in I_n} \tn \phi_{n-1}(t) \bfPhi(t_{n-1}) + \phi_{n}(t) \bfPhi(t_{n}) \tn_{A}
\\
&\leq
\sup_{t \in I_n} \bigl(
\phi_{n-1}(t) \tn  \bfPhi(t_{n-1})  \tn_{A}  + \phi_{n}(t) \tn \bfPhi(t_{n}) \tn_{A}
\bigr)
\\
&\leq
\tn  \bfPhi(t_{n-1})  \tn_{A}  + \tn \bfPhi(t_{n}) \tn_{A}
\label{eq:time-max-end}
\end{align}
and hence, via \eqref{eq:sldkgn}, we get the stability
\begin{align}
\| \bfPhi \|^2_{L^\infty(A)}
\lesssim
\tn \bfPsi \tn^2_{A}
\end{align}
For the remaining first term in \eqref{eq:dualstabest}, we look at a single time interval $I_n$. Since $\bfPhi_1$ is a discrete function, we can bound the max norm in time by the $L^1$-norm
\begin{align}
\|  \dot{\bfPhi}_1 \|_{L^\infty(I_n,H^{-1}(\Omega))}
&:=
\sup_{t \in I_n}
\left(
\sup_{\bfzero \neq \bfw \in \Vhnz }
\frac{(  \dot{\bfPhi}_1 , \bfw )}{\| \bfw \|_{H^1(\Omega)}}
\right)
\\
&=
\sup_{\bfzero \neq \bfw \in \Vhnz }
\left(
\sup_{t \in I_n}
\frac{(  \dot{\bfPhi}_1 , \bfw )}{\| \bfw \|_{H^1(\Omega)}}
\right)
\\
&\lesssim
\sup_{\bfzero \neq \bfw \in \Vhnz }
k^{-1}
\int_{I_n}
\frac{(  \dot{\bfPhi}_1 , \bfw )}{\| \bfw \|_{H^1(\Omega)}}
\, dt
\\&\lesssim k^{-1} \label{eq:boidf1}
\sup_{\bfzero \neq \bfw \in \Vhnz }
\frac{\int_{I_n}
( \rho \dot{\bfPhi}_1 , \bfw )
\, dt}{\| \bfw \|_{H^1(\Omega)}}
\\&= k^{-1}  \label{eq:boidf2}
\sup_{\bfzero \neq \bfw \in \Vhnz }
\frac{\int_{I_n}
a_E( \bfPhi_0 , \bfw ) + \sum_{m=1}^M a_{VE}^m( \bfPhi_{VE}^m , \bfw )
\, dt}{\| \bfw \|_{H^1(\Omega)}}
\\&\lesssim k^{-1}  \label{eq:boidf3}
\int_{I_n}
\tn \bfPhi_{0} \tn_E + \sum_{m=1}^M \tn \bfPhi_{VE}^m \tn_{VE,m}
\, dt
\\&\lesssim
\| \bfPhi_{0} \|_{L^\infty(I_n,E)} + \sum_{m=1}^M \| \bfPhi_{VE}^m \|_{L^\infty(I_n,VE)}
\end{align}
where we in \eqref{eq:boidf1} use the lower bound on $\rho > 0$,
in \eqref{eq:boidf2} use the discrete dual problem \eqref{eq:discdualprob},
and in \eqref{eq:boidf3} use the Cauchy--Schwarz inequality.
In turn this inequality gives
\begin{align}
\| \dot{\bfPhi}_1 \|_{L^\infty(H^{-1}(\Omega))}
\lesssim
\| \bfPhi \|^2_{L^\infty(A)}
\end{align}
which concludes the proof of the stability estimate required for the energy estimate.

\paragraph{\boldmath (iii) Stability for $L^2$ estimate.}
We express the discrete the dual problem \eqref{eq:discdualprob} as
\begin{equation}
\int_{I_n}
-( \mcA_h \dot{\bfPhi},\bfW ) + (\mcB_h \bfPhi,\bfW )
\, dt = 0 \,,
\qquad \forall \bfW \in [\Vhnz]^{2+M}
\end{equation}
Choosing $\bfW= \mcA_h^{-1} \mcB_h \dot{\bfPhi}$
and summing over all intervals $\{I_n\}_{i+1}^j$ we obtain
\begin{align}
0&= \int_{t_i}^{t_j}
-(\mcA_h \dot{\bfPhi},\mcA_h^{-1} \mcB_h \dot{\bfPhi}) + (\mcB_h \bfPhi,\mcA_h^{-1} \mcB_h \dot{\bfPhi} )
\, dt
\\
&= \int_{t_i}^{t_j} -(\dot{\bfPhi},\mcB_h \dot{\bfPhi}) + (\mcB_h \bfPhi,\mcA_h^{-1} \mcB_h \dot{\bfPhi} )
\, dt
\\
&= \int_{t_i}^{t_j} -B_{Sym}(\dot{\bfPhi},\dot{\bfPhi}) + (\mcB_h \bfPhi,\mcA_h^{-1} \mcB_h \dot{\bfPhi} )
\, dt
\\
&= \int_{t_i}^{t_j} -B_{Sym}(\dot{\bfPhi},\dot{\bfPhi}) + \frac{1}{2}\partial_t (\mcB_h \bfPhi,\mcA_h^{-1} \mcB_h \bfPhi )
\, dt
\end{align}
which, after integration gives
\begin{equation}
(\mcB_h \bfPhi(t_j),\mcA_h^{-1} \mcB_h {\bfPhi}(t_j) )
- 2 \int_{t_i}^{t_j} B_{Sym}(\dot{\bfPhi},\dot{\bfPhi}) 
= (\mcB_h \bfPhi(t_i),\mcA_h^{-1} \mcB_h {\bfPhi}(t_i) )
\end{equation}
and we have the identity 
\begin{align}
(\mcB_h \bfPhi,\mcA_h^{-1} \mcB_h {\bfPhi})
&= \rho^{-1}\bigl\| \mcL_{h,E} \bfPhi_0 + \sum_{m=1}^M \mcL_{h,VE}^m \bfPhi_{VE}^m  \bigr\|^2_{L^2(\Omega)}
\\&\quad\nonumber
+ \tn \bfPhi_1 \tn_E^2
+ \sum_{m=1}^M \bigl\tn -\bfPhi_1 + \frac{1}{\tau_m}\bfPhi_{VE}^m \bigr\tn^2_{VE,m}
\end{align}
By analogous arguments to the energy norm stability, using the piecewise linear interpolation in time, we
get the following stability
\begin{align}
&\rho^{-1}\bigl\| \mcL_{h,E} \bfPhi_0 + \sum_{m=1}^M \mcL_{h,VE}^m \bfPhi_{VE}^m  \bigr\|^2_{L^\infty(L^2(\Omega))}
+ \| \bfPhi_1 \|_{L^\infty(E)}^2
+ \sum_{m=1}^M \bigl\| -\bfPhi_1 + \frac{1}{\tau_m}\bfPhi_{VE}^m \bigr\|^2_{L^\infty([VE,m])}
\nonumber
\\
&\quad\leq
\rho^{-1}\bigl\| \mcL_{h,E} \bfPsi_0 + \sum_{m=1}^M \mcL_{h,VE}^m \bfPsi_{VE}^m  \bigr\|^2_{L^2(\Omega)}
+ \tn \bfPsi_1 \tn_E^2
+ \sum_{m=1}^M \bigl\tn -\bfPsi_1 + \frac{1}{\tau_m}\bfPsi_{VE}^m \bigr\tn^2_{VE,m}
\end{align}
The final form of the stability estimate comes by choosing $\bfPsi_1=\bfPsi_{VE}^m=\bfzero$ in the right hand side and recognizing that
the first three terms on the left in \eqref{eq:L2dualstabest} can be bounded from below via the following calculations.
Utilizing the discrete dual problem \eqref{eq:discdualprob} and the Cauchy--Schwarz inequality, we for the first term in the stability estimate have
\begin{align}
\| \dot{\bfPhi}_1 \|_{L^\infty(I_n,L^2(\Omega))}^2
&\lesssim
k^{-1} \int_{I_n}
(\rho \dot{\bfPhi}_1,\dot{\bfPhi}_1) \, dt
\\
&=
k^{-1} \int_{I_n}
(-\mcL_{h,E} \bfPhi_0 - \sum_{m=1}^M \mcL_{h,VE}^m \bfPhi_{VE}^m, \dot{\bfPhi}_1)
\, dt
\\
&\leq
k^{-1} \int_{I_n}
\bigl\| \mcL_{h,E} \bfPhi_0 + \sum_{m=1}^M \mcL_{h,VE}^m \bfPhi_{VE}^m \bigr\|_{L^2(\Omega)}
\| \dot{\bfPhi}_1 \|_{L^2(\Omega)}
\, dt
\\
&\leq
\bigl\| \mcL_{h,E} \bfPhi_0 + \sum_{m=1}^M \mcL_{h,VE}^m \bfPhi_{VE}^m \bigr\|_{L^\infty(I_n,L^2(\Omega))}
\| \dot{\bfPhi}_1 \|_{L^\infty(I_n,L^2(\Omega))}
\end{align}
which concludes the proof of the stability estimate.
\end{proof}

We now turn to the proof of our main a priori error estimates.
\subsection{A Priori Error Estimate}
\begin{thm}[End Time Error Bounds] \label{eq:energyestimate}
Assuming $\bfg_k:=\pi_k\bfg$, and that the velocities of the exact solution $\dot{\bfu}_0(t) = \bfu_1(t) \in [H^r(\Omega)]^3$ at all times $t\in (0,T]$, the following end time a priori error estimates hold. In the energy norm, we have the estimate
\begin{align}
\boxed{
\tn \bfu(T) - \bfU(T) \tn_A
\lesssim T
\bigl(
h^{s-1} \|\bfu_1 \|_{L^\infty(H^s(\Omega))}
+
k^2 \| \partial_{tt} \bff \|_{L^\infty(H^1(\Omega))}
\bigr)
}
\end{align}
and for displacements in $L^2$ norm we have the estimate
\begin{align}
\boxed{
\| \bfu_0(T)-\bfU_0(T) \|_{L^2(\Omega)}
\lesssim T
\bigl(
(h^{s}+kh^{s-1}) \|\bfu_1 \|_{L^\infty(H^s(\Omega))}
+ k^2 \| \partial_{tt} \bff \|_{L^\infty(L^{2}(\Omega))}
\bigr)
}
\end{align}
where $s=\min(r,p+1)$ and the constant in $\lesssim$ for each estimate is independent of mesh size $h$, timestep $k$, and end time $T$.
\end{thm}

\begin{rem}
Through some adjustments to the proof,
this theorem can be extended from bounds of the errors at the end time $T$ to bounds on the maximum of the errors at each nodal point in time.
\end{rem}

\begin{rem}
We explicitly state the dependence on the end time $T$ in the inequality to emphasize that this dependence is only linear. This is in contrast to some previous results cited in the introduction, where the constants are
exponentially dependent on $T$ due to the use of Grönwall type inequalities.
\end{rem}

\begin{rem}
The $k h^{s-1}$-term in the $L^2$ estimate is of optimal order if $k\sim h$. This term vanishes if the viscoelastic stresses would be formulated in terms of the same differential operator as the elastic stress.
\end{rem}

\begin{proof}
The proofs of the energy estimate and the $L^2$ estimate are presented in parallel, term by term, with the fundamental difference being which stability estimate is used.
\paragraph{Energy Estimate.}
Adding and subtracting terms in combination with the triangle inequality gives
\begin{align}
\tn \bfu(T) - \bfU(T) \tn_A
&\leq
\tn (\bfu - \pi_{hk}\bfu)(T) \tn_A
+
\tn (\pi_{hk}\bfu - \bfU)(T) \tn_A
\\&=
\tn (I - \pi_{h}) \bfu(T) \tn_A
+
\tn \bfeta(T) \tn_A
\end{align}
where the first term is limited through the interpolation estimate in Lemma~\ref{lemma:interpolation}.
For the second term we
choose $\bfPsi = \bfeta(T)$ in the error representation 
formula \eqref{eq:errorrep} which gives
\begin{align}
\tn \bfeta(T) \tn_A^2 &=
A(\bfeta(T),\bfPsi)
\\&=  \underbrace{A(\bfeta(0),\bfPhi(0))}_{=0} 
-
\underbrace{\smashoperator[r]{\int_0^T} A(\dot{\bfrho},\cproj \bfPhi) +  B(\bfrho,\cproj \bfPhi)
 - (\bfg - \bfg_k, P_0\bfPhi_1)_{\Gamma_N} \, dt}_{\bigstar}
\end{align}
where the first term is zero due to the definition of initial 
data $\bfeta(0)$.

\paragraph{\boldmath $L^2$ Estimate.}
Adding and subtracting terms and applying the triangle inequality gives
\begin{align}
\| \bfu_0(T) - \bfU_0(T) \|_{L^2(\Omega)}
&\leq
\| (I - R_{E}) \bfu_0(T) \|_{L^2(\Omega)}
+
\| \bfeta_0(T) \|_{L^2(\Omega)}
\end{align}
where the first term is limited through a standard $L^2$ error estimate for linear elastostatics.
For the second term we choose
$\bfPsi = (\bfPsi_1,\bfPsi_0,\bfPsi_{VE})= (\bfzero,\mcL_{h,E}^{-1}\bfeta_0,\bfzero)$ in the error representation formula \eqref{eq:errorrep} which gives
\begin{align}
\| \bfeta_0(T) \|_{L^2(\Omega)}^2
&=
A(\bfeta(T),\bfPsi)
\\&=  \underbrace{A(\bfeta(0),\bfPhi(0))}_{=0}
- \underbrace{\smashoperator[r]{\int_0^T} A(\dot{\bfrho},\cproj \bfPhi) +  B(\bfrho,\cproj \bfPhi)
 - (\bfg - \bfg_k, P_0\bfPhi_1)_{\Gamma_N} \, dt}_{\bigstar}
\end{align}
where the first term is zero due to the definition of initial 
data $\bfeta(0)$.
Note that the expression on the right is the same as in the energy estimate.

\paragraph{Terms in $\bigstar$.} We decompose $\bigstar$ into a sum of the following three terms
\begin{align}
I = \int_0^T A(\dot{\bfrho},\cproj \bfPhi) \, dt
,\quad
II = \int_0^T B(\bfrho,\cproj \bfPhi) \, dt
,\quad
III = -\int_0^T (\bfg - \bfg_k, P_0\bfPhi_1)_{\Gamma_N} \, dt
\end{align}


\paragraph{\boldmath Term $I$.} We have the identity
\begin{align}
\int_0^T A(\dot{\bfrho},P_0 \bfPhi) \, dt
&= \label{eq:termI}
\sum_{n=1}^N \int_{I_n} (\rho \dot{\bfrho}_1,P_0 \bfPhi_1) \, dt
\\&\qquad\quad +  \underbrace{\int_{I_n} a_E(\dot{\bfrho}_0,P_0 \bfPhi_0)}_{=0} \, dt
+ \sum_{m=1}^M 
\underbrace{\int_{I_n} 
a_{VE}^m( \dot{\bfrho}_{VE}^m,P_0 \bfPhi^m_{VE})}_{=0} \, dt
\nonumber
\end{align}
where the second term is zero since
\begin{equation}
\int_{I_n} a_E(\dot{\bfrho}_0,P_0 \bfPhi_0) \, dt =  
a_E(\bfrho_0(t_n),P_0 \bfPhi_0(t_n)^-) 
- a_E(\bfrho_0(t_{n-1}),P_0 \bfPhi_0(t_{n-1})^+)
=0
\label{eq:dogij}
\end{equation}
due to the definition of the interpolant using Ritz projections.
We can argue in the same way for the third term. 
For the remaining first term in \eqref{eq:termI} we have
\begin{align}
\sum_{n=1}^{N} \int_{I_n} (\rho\dot{\bfrho}_1,P_0 \bfPhi_1) 
&=-(\rho\bfrho_1(0),\bfPhi_1 (0) ) + (\rho\bfrho_1(0),(I - P_0)\bfPhi_1 (0) )
\\ \nonumber
&\qquad + (\rho\bfrho_1(T), \bfPhi_1 (T) ) - (\rho\bfrho_1(T),(I-P_0) \bfPhi_1 (T))
\\ \nonumber
&\qquad  
\underbrace{- \sum_{n=1}^{N-1} (\rho\bfrho_1(t_n),[P_0 \bfPhi_1]_n)_{\Omega}}_{I_3}
\\
&=I_1 + I_2 + I_3
\end{align}

\paragraph{\boldmath Terms $I_1$ and $I_2$: Energy Estimate.} Using Hölder's inequality we have 
\begin{align}
|I_1|&= |(\rho\bfrho_1(0),\bfPhi_1 (0) ) + (\rho\bfrho_1(0),(I - P_0)\bfPhi_1 (0) )|
\\
&\lesssim 
\|\bfrho_1(0)\|_{L^2(\Omega)} \|\bfPhi_1 (0)\|_{L^2(\Omega)} 
+ \|\bfrho_1(0)\|_{H^1(\Omega)} \|(I - P_0)\bfPhi_1 (0) )\|_{H^{-1}(\Omega)}
\\
&\lesssim 
h^{s}\| \bfu_1(0)\|_{H^s(\Omega)} \|\bfPhi_1 \|_{L^\infty(I_1,L^2(\Omega))} 
+ h^{s-1} k \|\bfu_1(0)\|_{H^s(\Omega)} \|\dot{\bfPhi}_1\|_{L^\infty(I_1,H^{-1}(\Omega))}
\end{align}
and $I_2$ is estimated in the same way.

\paragraph{\boldmath Terms $I_1$ and $I_2$: $L^2$ Estimate.} Using Hölder's inequality followed by Korn's inequality we have  
\begin{align}
|I_1|&= |(\rho\bfrho(0),\bfPhi_1 (0) ) + (\rho\bfrho(0),(I - P_0)\bfPhi_1 (0) )|
\\
&\lesssim 
\underbrace{\|\bfrho_1(0)\|_{H^{-1}(\Omega)}}_{\leq \|\bfrho_1(0)\|_{L^{2}(\Omega)} }
\underbrace{\|\bfPhi_1 (0)\|_{H^1(\Omega)}}_{\lesssim \tn\bfPhi_1 (0)\tn_E }
+ \|\bfrho_1(0)\|_{L^2(\Omega)} \|(I - P_0)\bfPhi_1 (0) )\|_{L^2(\Omega)}
\\
&\lesssim 
h^{s} \| \bfu_1(0)\|_{H^s(\Omega)} \|\bfPhi_1 \|_{L^\infty(I_1,E)} 
+ h^{s} \|\bfu_1(0)\|_{H^s(\Omega)} k \|\dot{\bfPhi}_1\|_{L^\infty(I_1,L^2(\Omega))}
\end{align}
and $I_2$ is estimated in the same way.

\paragraph{\boldmath Terms $I_3$: Energy Estimate.} Using Hölder's inequality, we have
\begin{align}
| I_3 | &\lesssim \sum_{n=1}^{N-1} \|\bfrho_1(t_n)\|_{H^1(\Omega)}
\bigl\|[P_0 \bfPhi_1]_n\bigr\|_{H^{-1}(\Omega)}
\\
&\lesssim \sum_{n=1}^{N-1} k h^{s-1} \|\bfu_1(t_n) \|_{H^s(\Omega)}
\| \dot{\bfPhi}_1 \|_{L^\infty(I,H^{-1}(\Omega))}
\\
&\lesssim h^{s-1} \left(\sum_{n=1}^{N-1} k \|\bfu_1(t_n) \|_{H^s(\Omega)}\right)
\| \dot{\bfPhi}_1 \|_{L^\infty(I,H^{-1}(\Omega))}
\end{align}
where we used the estimate 
\begin{align}
\|[P_0 \bfPhi_1]_n\|_{H^{-1}(\Omega)}
&\leq 
\|((I-P_0)\bfPhi_1)^+_n\|_{H^{-1}(\Omega)}
+ \|((I-P_0) \bfPhi_1)^-_n\|_{H^{-1}(\Omega)}
\\
&\lesssim 
k \|\dot{\bfPhi}_1\|_{L^\infty(I_{n+1},H^{-1}(\Omega))}
+ k \|\dot{\bfPhi}_1\|_{L^\infty (I_n, H^{-1}(\Omega))}
\end{align}

\paragraph{\boldmath Terms $I_3$: $L^2$ Estimate.} 
 Using Hölder's inequality, we have
\begin{align}
| I_3 | &\lesssim \sum_{n=1}^{N-1} \|\bfrho_1(t_n)\|_{L^2(\Omega)}
\bigl\|[P_0 \bfPhi_1]_n \bigr\|_{L^2(\Omega)}
\\
&\lesssim \sum_{n=1}^{N-1} k h^{s} \|\bfu_1(t_n) \|_{H^s(\Omega)}
\| \dot{\bfPhi}_1 \|_{L^\infty(I,L^2(\Omega))}
\\
&\lesssim h^{s} \left(\sum_{n=1}^{N-1} k \|\bfu_1(t_n) \|_{H^s(\Omega)}\right)
\| \dot{\bfPhi}_1 \|_{L^\infty(I,L^2(\Omega))}
\end{align}
where we used the estimate 
\begin{align}
\|[P_0 \bfPhi_1]_n\|_{L^2(\Omega)}
&\leq 
\|((I-P_0)\bfPhi_1)^+_n\|_{L^2(\Omega)}
+ \|((I-P_0) \bfPhi_1)^-_n\|_{L^2(\Omega)}
\\
&\lesssim 
k \|\dot{\bfPhi}_1\|_{L^\infty(I_{n+1},L^2(\Omega))}
+ k \|\dot{\bfPhi}_1\|_{L^\infty (I_n, L^2(\Omega))}
\end{align}

\paragraph{\boldmath Final Estimates Term $I$: Energy Estimate.}
Collecting the estimates we have
\begin{align}
|I| &\lesssim h^{s} \Big( \|\bfu_1(0) \|_{H^s(\Omega)} 
+ \|\bfu_1(T) \|_{H^s(\Omega)} \Big) 
\|\bfPhi_1\|_{L^\infty(I,L^2(\Omega))}
\\ \nonumber
&\qquad + h^{s-1} \, T \max_{0 \leq n \leq N} \|\bfu_1(t_n) \|_{H^s(\Omega)}
\| \dot{\bfPhi}_1 \|_{L^\infty(I,H^{-1}(\Omega))}
\end{align}

\paragraph{\boldmath Final Estimates Term $I$: $L^2$ Estimate.}
Collecting the estimates we have
\begin{align}
|I| &\lesssim h^{s} \Big( \|\bfu_1(0) \|_{H^s(\Omega)} 
+ \|\bfu_1(T) \|_{H^s(\Omega)} \Big) 
\|\bfPhi_1\|_{L^\infty(I,H^1(\Omega))}
\\ \nonumber
&\qquad + h^{s} \, T \max_{0 \leq n \leq N} \|\bfu_1(t_n) \|_{H^s(\Omega)}
\| \dot{\bfPhi}_1 \|_{L^\infty(I,L^2(\Omega))}
\end{align}


\paragraph{\boldmath Term $II$.} We have the identity 
\begin{align}
\int_0^T  B( {\bfrho},P_0 \bfPhi) \, dt
&=\int_0^T {a_E(\bfrho_0,P_0 \bfPhi_1)}   \, dt
- \int_0^T  {a_E(\bfrho_1,P_0 \bfPhi_0)}  \, dt
\\ \nonumber
&\qquad + \sum_{m=1}^M \int_0^T  
a_{VE}^m(\bfrho_{VE}^m,P_0\bfPhi_1)
- a_{VE}^m(\bfrho_1,P_0\bfPhi^m_{VE})  \, dt
\\ \nonumber
&\qquad +\sum_{m=1}^M \frac{1}{\tau_m}\int_0^T 
a_{VE}^m( \bfrho_{VE}^m,P_0\bfPhi^m_{VE})  \, dt
\\
&=\sum_{i=1}^5 II_i
\end{align}

For the remaining terms, we have the following estimates.
\paragraph{\boldmath Terms $II_1$, $II_3$ and $III$.} For term $II_1$, we by the definition of the Ritz projection have
\begin{align}
II_1
&= 
\int_0^T a_E((I - \pi_k R_E) \bfu_0 ,P_0 \bfPhi_1) \, dt
\\
&= 
\int_0^T \underbrace{a_E((I - R_E) \bfu_0 ,P_0 \bfPhi_1)}_{=0} \, dt
+ 
\int_0^T a_E(R_E (I- \pi_k) \bfu_0 ,P_0 \bfPhi_1) \, dt
\\
&= 
\int_0^T a_E( (I- \pi_k) \bfu_0 ,P_0 \bfPhi_1) \, dt
\end{align}
and the corresponding calculation for term $II_3$ yields
\begin{align}
II_3
&= 
\int_0^T \sum_{m=1}^M
a_{VE}^m((I - \pi_k R_{VE}) \bfu_{VE}^m ,P_0 \bfPhi_1) \, dt
\\
&= 
\int_0^T \sum_{m=1}^M
a_{VE}^m( (I- \pi_k) \bfu_{VE}^m ,P_0 \bfPhi_1) \, dt
\end{align}
Recalling the identity
\begin{align}
(\bfsig(\bfv_0,\bfv_{VE}), \bfeps(\bfw)) = a_E( \bfv_0 , \bfw ) + \sum_{m=1}^M
a_{VE}^m( \bfv_{VE}^m ,\bfw)
\end{align}
and that $\bfsig$ is a linear operator,
we after integration by parts have
\begin{align}
II_1 + II_3 + III &=
\int_0^T  ((I- \pi_k)\bfg, P_0 \bfPhi_1)_{\Gamma_N} \, dt
\\&\quad - \nonumber
\int_{0}^T  ( (I- \pi_k)(\bfsig(\bfu_0,\bfu_{VE}) \cdot \nabla), P_0 \bfPhi_1 ) \, dt
\\&\quad \nonumber
- \int_0^T (\bfg - \bfg_k, P_0 \bfPhi_1)_{\Gamma_N} \, dt
\end{align}
Note that the Neumann boundary terms vanish by choosing $\bfg_k:=\pi_k\bfg$.
Adding
\begin{align} \label{eq:zeroterm}
0 = \int_{I_n} (\partial_t (I - \pi_k)\rho\bfu_1, \cproj\bfPhi_1 ) \, dt =
\int_{I_n} ((I - \pi_k)\rho\dot{\bfu}_1, \cproj\bfPhi_1 ) \, dt
\end{align}
which is zero by the same arguments as \eqref{eq:dogij} and where the second equality holds as a consequence of
the definition of the time interpolant $\pi_k$ by the following calculation
\begin{align}
\int_{I_n} \partial_t \pi_k \bfv \, dt
&=
\pi_k \bfv(t_n) - \pi_k \bfv(t_{n-1})
=
\bfv(t_n) - \bfv(t_{n-1})
=
\int_{I_n} \partial_t \bfv \, dt
\label{eq:dotrho}
\end{align}
Assuming $\bfg_k:=\pi_k\bfg$ we thus have
\begin{align}
II_1 + II_3 + III
&=
\int_0^T  ((I- \pi_k)(\underbrace{\rho\dot{\bfu}_1 - \bfsig(\bfu_0,\bfu_{VE}) \cdot \nabla}_{=\bff \text{ by \eqref{eq:strong1}}}), P_0 \bfPhi_1 ) \, dt
\\&=
\int_0^T ((I- \pi_k)\bff, \bfPhi_1 ) \, dt + \int_0^T ((I- \pi_k)\bff, (P_0-I) \bfPhi_1 ) \, dt
\end{align}

\paragraph{\boldmath Final Estimates Term $(II_1 + II_3 + III)$: Energy Estimate.}
By the Cauchy--Schwarz and Hölder's inequalities, we have
\begin{align}
| II_1 + II_3 + III |&\lesssim
T \| (I- \pi_k) \bff \|_{L^\infty(L^2(\Omega))}
\| \bfPhi_1 \|_{L^\infty(L^2(\Omega))}
\\&\quad+ \nonumber
T \| (I- \pi_k) \bff \|_{L^\infty(H^1(\Omega))}
\| (I-\cproj) \bfPhi_1 \|_{L^\infty(H^{-1}(\Omega))}
\\&\lesssim
k^2  T \| \partial_{tt} \bff \|_{L^\infty(H^1(\Omega))}
\left(
\| \bfPhi_1 \|_{L^\infty(L^2(\Omega))}
+
k \| \dot{\bfPhi}_1 \|_{L^\infty(H^{-1}(\Omega))}
\right)
\end{align}

\paragraph{\boldmath Final Estimates Term $(II_1 + II_3 + III)$:  $L^2$ Estimate.}
By Hölder's inequality and the Cauchy-Schwarz inequality, we have
\begin{align}
| II_1 + II_3 + III | &\lesssim
T\| (I- \pi_k) \bff \|_{L^\infty(H^{-1}(\Omega))}
\| \bfPhi_1 \|_{L^\infty(H^1(\Omega))}
\\&\quad+ \nonumber
T\| (I- \pi_k) \bff \|_{L^\infty(L^2(\Omega))}
\| (I - \cproj) \bfPhi_1 \|_{L^\infty(L^2(\Omega))}
\\&\lesssim
k^2 T \| \partial_{tt} \bff \|_{L^\infty(L^2(\Omega))}
\left(
\| \bfPhi_1 \|_{L^\infty(E)}
+
k \| \dot{\bfPhi}_1 \|_{L^\infty(L^2(\Omega))}
\right)
\end{align}

\paragraph{\boldmath Term $II_2$.} This term is zero by the following calculation
\begin{align}
\int_{I_n} a_E(\bfrho_1,P_0 \bfPhi_0) \, dt
&= 
\int_{I_n} a_E((I - \pi_k R_E) \bfu_1 ,P_0 \bfPhi_0) \, dt
\\
&= 
\int_{I_n} a_E( (I- \pi_k) \bfu_1 ,P_0 \bfPhi_0) \, dt
\\
&= \label{eq:obduf}
\int_{I_n} a_E( (I- \pi_k) \bfu_1 ,P_0 \bfPhi_0) \, dt
-
\underbrace{
\int_{I_n} a_E( (I- \pi_k) \dot{\bfu}_0 ,P_0 \bfPhi_0) \, dt
}_{\text{$=0$ analogously to \eqref{eq:zeroterm}}}
\\
&= 
\int_{I_n} 
a_E( (I- \pi_k) \underbrace{( \bfu_1 - \dot{\bfu}_0)}_{=0 \text{ by \eqref{eq:strong2}}} , P_0 \bfPhi_0 ) \, dt
=0
\end{align}
where we, in the first equality, used the definition of the Ritz projection.

\paragraph{Terms $\bfI\bfI_4$ and $\bfI\bfI_5$.}
First looking at $II_4$, we have terms
\begin{align}
-\int_{I_n} a_{VE}^m(\bfrho_1, P_0 \bfPhi_{VE}^m) \, dt
&=
-\int_{I_n} a_{VE}^m((I - \pi_k R_{E})\bfu_1, P_0 \bfPhi_{VE}^m) \, dt
\\&=
-\int_{I_n} a_{VE}^m((I - \pi_k)\bfu_1, P_0 \bfPhi_{VE}^m) \, dt
\label{eq:ibduff}
\\&\quad \nonumber
-\int_{I_n} a_{VE}^m(\pi_k(I-R_E)\bfu_1, P_0 \bfPhi_{VE}^m) \, dt
\end{align}
In this case, the last term is not zero by the definition of the Ritz projections since $R_E$ is applied within a viscoelastic form.
Correspondingly, for $II_5$ we have
\begin{align}
\frac{1}{\tau_m} 
\int_{I_n}
a_{VE}^m( \bfrho_{VE}^m,P_0\bfPhi^m_{VE}) \, dt
&=
\frac{1}{\tau_m}
\int_{I_n} 
a_{VE}^m( (I - \pi_k R_{VE})\bfu_{VE}^m,P_0\bfPhi^m_{VE}) \, dt
\\&=
\frac{1}{\tau_m}
\int_{I_n} 
a_{VE}^m( (I - \pi_k)\bfu_{VE}^m,P_0\bfPhi^m_{VE}) \, dt
\label{eq:lkjsgn}
\end{align}
Pairing \eqref{eq:lkjsgn} with the first term in \eqref{eq:ibduff}, we see that the resulting term vanishes by the following calculation
\begin{align}
&\frac{1}{\tau_m}\int_{I_n} 
a_{VE}^m( (I - \pi_k)\bfu_{VE}^m,P_0\bfPhi^m_{VE}) \, dt
-\int_{I_n} a_{VE}^m((I - \pi_k)\bfu_1, P_0 \bfPhi_{VE}^m) \, dt
\\&\qquad=
\underbrace{\int_{I_n} a_{VE}^m((I - \pi_k)\dot{\bfu}_{VE}^m, P_0 \bfPhi_{VE}^m) \, dt}_{\text{$=0$ analogously to \eqref{eq:zeroterm}}}
\\&\nonumber\qquad\qquad
+
 \frac{1}{\tau_m}\int_{I_n} 
a_{VE}^m( (I - \pi_k)\bfu_{VE}^m,P_0\bfPhi^m_{VE}) \, dt
-\int_{I_n} a_{VE}^m((I - \pi_k)\bfu_1, P_0 \bfPhi_{VE}^m) \, dt
\\&\qquad=
\int_{I_n} a_{VE}^m \bigl((I - \pi_k)\underbrace{\bigl( \bfu_1 + \frac{1}{\tau_m}\bfu_{VE}^m- \dot{\bfu}_{VE}^m \bigr)}_{\text{$=0$ by \eqref{eq:strong3}}}, P_0 \bfPhi_{VE}^m \bigr) \, dt
= 0
\end{align}
Thus, $II_4+II_5$ is reduced to
\begin{align}
II_4+II_5
&=
\sum_{m=1}^M \int_{0}^T a_{VE}^m(\pi_k(I-R_E)\bfu_1, \cproj \bfPhi_{VE}^m) \, dt
\\
&=
\sum_{m=1}^M \int_{0}^T a_{VE}^m((I-R_E)\bfu_1, \cproj \bfPhi_{VE}^m) \, dt
\\\nonumber&\qquad\quad
-\int_{0}^T a_{VE}^m((I-\pi_k)(I-R_E)\bfu_1, \cproj \bfPhi_{VE}^m) \, dt
\\\nonumber&\qquad\quad
+\underbrace{\int_{0}^T a_{VE}^m((I-\pi_k)\partial_t(I-R_E)\bfu_0, \cproj \bfPhi_{VE}^m) \, dt}_{\text{$=0$ analogously to \eqref{eq:zeroterm}}}
\\ \label{eq:fbjbba}
&=
\sum_{m=1}^M \int_{0}^T a_{VE}^m((I-R_E)\bfu_1, \cproj \bfPhi_{VE}^m) \, dt
\\\nonumber&\qquad\quad
-\int_{0}^T a_{VE}^m((I-\pi_k)(I-R_E)(\underbrace{\bfu_1-\dot{\bfu}_0}_{\mathclap{=0 \text{ by \eqref{eq:strong2}}}}), \cproj \bfPhi_{VE}^m) \, dt
\\&=
\sum_{m=1}^M \int_{0}^T a_{VE}^m((I-R_E)\bfu_1, \cproj \bfPhi_{VE}^m) \, dt
\label{eq:fbjbbb}
\end{align}
where we in \eqref{eq:fbjbba} use that the Ritz projection $R_E$ commutes with the time derivative since the domain $\Omega$ does not change.

Note that in case the viscoelastic stresses would be expressed in terms of the same differential operator as the elastic stress, the term $II_4+II_5$ would vanish by the definition of the Ritz projection $R_E$.

\paragraph{\boldmath Final Estimates Term $(II_4 + II_5)$: Energy Estimate.}
Starting from \eqref{eq:fbjbbb}, using the triangle inequality, and applying the Cauchy-Schwarz inequality in space, and then in time, we have
\begin{align}
| II_4 + II_5 | &=
\biggl|\sum_{m=1}^M \int_{0}^T a_{VE}^m((I-R_E)\bfu_1, \cproj \bfPhi_{VE}^m) \, dt \biggr|
\\&\leq
\sum_{m=1}^M \int_{0}^T \underbrace{\tn (I-R_E)\bfu_1 \tn_{VE,m}}_{\lesssim \kappa_m \| (I-R_E)\bfu_1 \|_{H^1(\Omega)}} \tn \cproj \bfPhi_{VE}^m \tn_{VE,m} \, dt
\\&\lesssim
\int_{0}^T \| (I-R_E)\bfu_1 \|_{H^1(\Omega)}
\left( \sum_{m=1}^M \kappa_m \tau_m \right)^{1/2}
\left( \sum_{m=1}^M \frac{2}{\tau_m} \tn \cproj \bfPhi_{VE}^m \tn_{VE,m}^2\right)^{1/2} \, dt
\\&\lesssim
\int_{0}^T h^{s-1} \| \bfu_1 \|_{H^s(\Omega)}
\left( \sum_{m=1}^M \frac{2}{\tau_m} \tn \cproj \bfPhi_{VE}^m \tn_{VE,m}^2\right)^{1/2} \, dt
\\&\lesssim
h^{s-1}
\left( \int_{0}^T \| \bfu_1 \|_{H^s(\Omega)}^2 \, dt \right)^{1/2}
\left( \sum_{m=1}^M \frac{2}{\tau_m} \int_{0}^T \tn \cproj \bfPhi_{VE}^m \tn_{VE,m}^2 \, dt \right)^{1/2}
\\&\lesssim
h^{s-1} \sqrt{T}
\| \bfu_1 \|_{L^\infty(H^s(\Omega))}
\tn \bfPsi \tn_A
\end{align}
where we note that $\sqrt{T}\leq T$ assuming $T\geq 1$.

\paragraph{\boldmath Final Estimates Term $(II_4 + II_5)$:  $L^2$ Estimate.}
Due to more technical and lengthy arguments, we for this term prove the bound in Appendix~\ref{sec:appendix}.
Starting from \eqref{eq:fbjbbb} and using estimate \eqref{eq:technical-result} we get
\begin{align}
\label{eq:complicated-term}
|II_4 + II_5| &=
\biggl|
\sum_{m=1}^M a_{VE}^m( (I-R_E)\bfu_1, \cproj \bfPhi_{VE}^m) \, dt
\biggr|
\\
&\lesssim T
(h^{s} + k h^{s-1}) \| \bfu_1 \|_{L^\infty(H^s(\Omega))} \|\mcL_{h,E} \bfPsi_0 \|_{L^2(\Omega)}
\end{align}
which concludes the proof.

\end{proof}


\section{Numerical Results}
\label{sect:numerical}

To illustrate the method and its analysis, we in this section provide results from numerical experiments. Unless otherwise stated, SI-units are assumed.

\subsection{Convergence} \label{sect:convergence}

To confirm our a priori error estimates, we manufacture a problem with a known analytical solution by making the following ansatz on the velocity field
\begin{align}
\dot{\bfu}_0(t;x,y,z) := \exp(1-t)(a_1 t + a_2 t^2)
\begin{bmatrix}
\frac{3}{4} \cos(\pi x)  \sin\left( \frac{\pi( 2 y+1)}{4} \right) \sin(\pi z) \\
\frac{3}{4} \sin\left( \frac{\pi( 2 x+1)}{4} \right) \cos(\pi y) \sin(\pi z) \\
\sin\left( \frac{\pi( 2 x+1)}{4} \right) \sin\left( \frac{\pi( 2 y+1)}{4} \right) \sin\left(\frac{\pi z}{2}\right)
\end{bmatrix}
\label{eq:vansatz}
\end{align}
with $a_1=a_2=\frac{1}{5}$. The computational domain $\Omega \times [0,T]$ for the problem is the unit cube $\Omega = [0,1]^3$ and end time $T=1 $. On the cube bottom, $z=0$, we have the homogenous Dirichlet condition $\bfu_0=\bfzero$ and on the remainder of the boundary, we have non-homogenous Neumann conditions manufactured from the ansatz \eqref{eq:vansatz}.
We choose to simulate a viscoelastic material featuring an elastic part and a single Maxwell arm, i.e., the standard linear solid model. As material parameters we choose density $\rho = 100$, Poisson ratio $\nu = 0.3$, $E$- and Maxwell arm stiffness moduli $E = \kappa = 10^5$, and relaxation time $\tau = 10^3/\kappa$.

We here present simulation results using standard Lagrange $P_3$ finite elements on tetrahedra for the spatial discretization and piecewise linears for the time integration (Crank--Nicolson). The reason for using finite elements with as high polynomial approximation as $p=3$ is to generate cases where time integration is the dominant source of error. To confirm our a priori estimates in Theorem \ref{eq:energyestimate}, we perform numerical simulations for the model problem using timesteps ranging from $k=1$ to $k=1/128$ and mesh sizes ranging from $h=1$ to $h=1/5$.

\paragraph{Convergence in Energy Norm.}
In Figure~\ref{fig:Eerror_vs_h} we plot the error in energy norm against the mesh size for each choice of timestep $k$. As the error in the time integration becomes dominant when the mesh size decreases, the curves level out.
The curve for the shortest timestep $k=1/128$ does not level out in this range of mesh, and thus, we assume the error in the space discretization is dominant, indicating a convergence rate of $h^p$ in agreement with the energy estimate in Theorem~\ref{eq:energyestimate}.
In Figure~\ref{fig:Eerror_vs_k}, we make the corresponding plot of the error in energy norm against the timestep for each choice of mesh size $h$. This indicates a convergence rate of $k^2$ in regions where
the time integration dominates the error. In conclusion, the numerical study of convergence in energy norm
gives support for an estimate of the form
\begin{align}
\tn \bfu(T) - \bfU(T) \tn_A \lesssim h^p + k^2
\end{align}
as is the energy estimate in Theorem~\ref{eq:energyestimate}.

\begin{figure}
\centering
\includegraphics[width=0.70\linewidth]{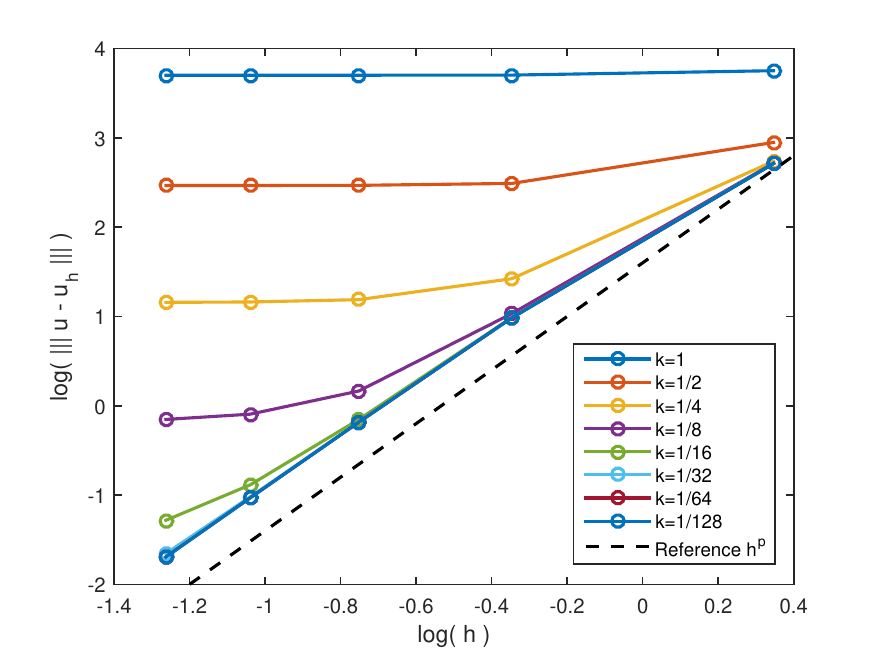}
\caption{Convergence of model problem end time error in energy norm with respect to the mesh size $h$.
For the finite element approximation, standard Lagrange $P_3$ elements on tetrahedra were used ($p=3$).}
\label{fig:Eerror_vs_h}
\end{figure}

\begin{figure}
\centering
\includegraphics[width=0.70\linewidth]{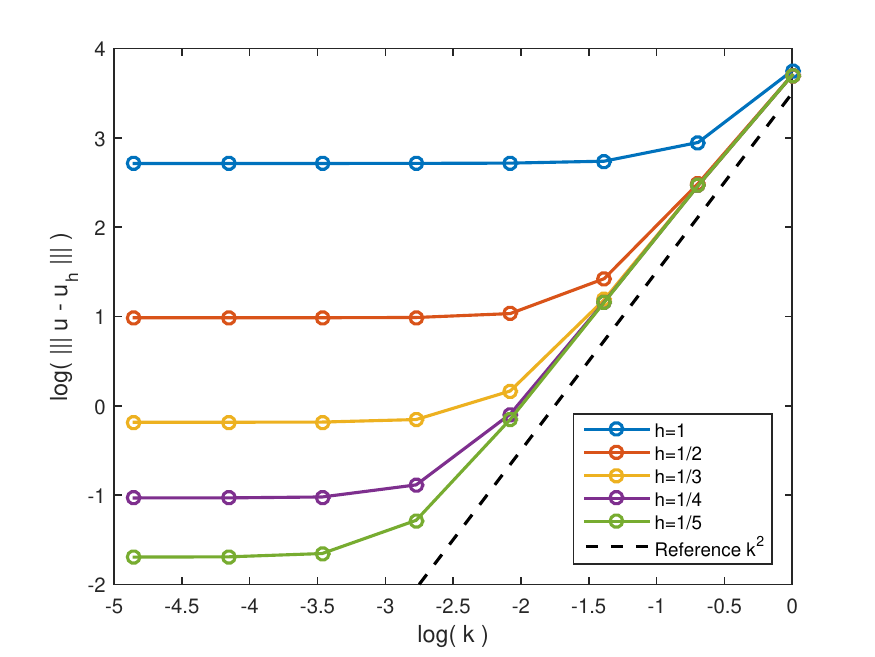}
\caption{Convergence of model problem end time error in energy norm with respect to the timestep $k$.
For the finite element approximation, standard Lagrange $P_3$ elements on tetrahedra were used ($p=3$).}
\label{fig:Eerror_vs_k}
\end{figure}

\paragraph{Convergence in ${\boldsymbol L}^2$-norm.}
Analogously to the study of the end time error in energy norm we in
Figure~\ref{fig:L2error_vs_h} and Figure~\ref{fig:L2error_vs_k} provide numerical results for the end time error of the displacements in $L^2$ norm.
In agreement with Theorem \ref{eq:energyestimate} these results
indicate that the end time $L^2$-error estimate for the displacement field should be of the form
\begin{align}
\| \bfu_0(T) - \bfU_0(T) \|_{L^2(\Omega)} \lesssim h^{p+1} + k^2
\end{align}
We do not clearly observe the additional $k h^p$ scaling also implied by the estimate. However, such a term would need a very large constant to be noticable in the ranges of $h$ and $k$, when either $k^2$ is dominating the error, or $h^{p+1}$ is dominating the error.

\begin{figure}
\centering
\includegraphics[width=0.70\linewidth]{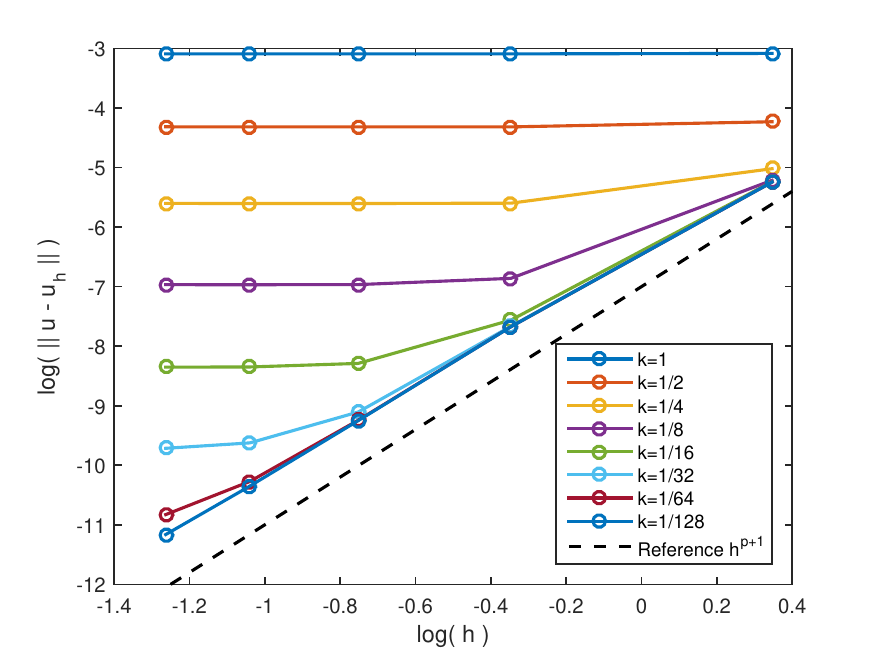}
\caption{Convergence of model problem end time $L^2(\Omega)$-error with respect to the mesh size $h$.
For the finite element approximation, standard Lagrange $P_3$ elements on tetrahedra were used ($p=3$).}
\label{fig:L2error_vs_h}
\end{figure}

\begin{figure}
\centering
\includegraphics[width=0.70\linewidth]{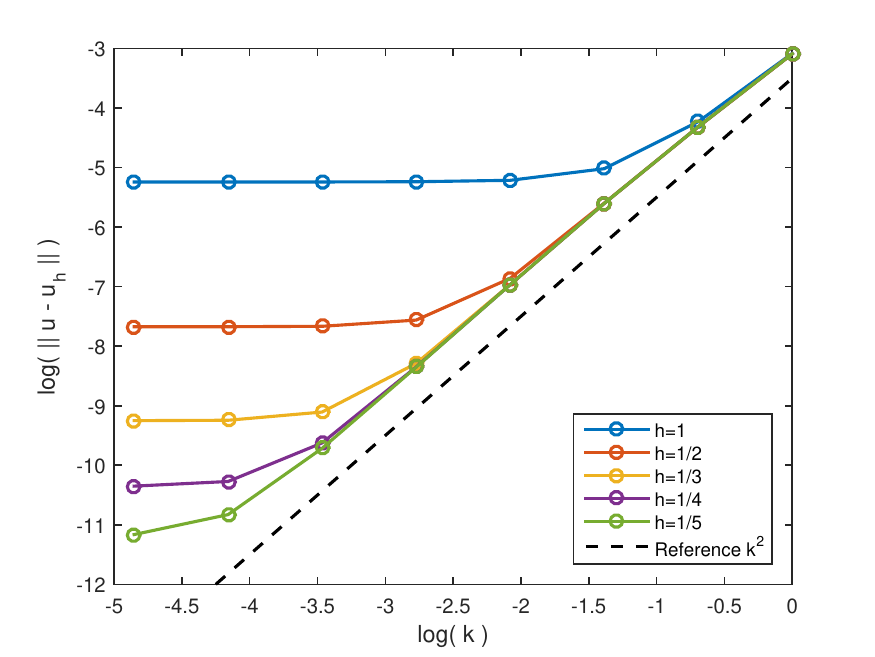}
\caption{Convergence of model problem end time $L^2(\Omega)$-error with respect to the timestep $k$.
For the finite element approximation, standard Lagrange $P_3$ elements on tetrahedra were used ($p=3$).}
\label{fig:L2error_vs_k}
\end{figure}

\subsection{Energy Conservation}

To illustrate Lemma~\ref{lemma:disccons}, the discrete conservation law, we consider the same cube and material as specified in Section~\ref{sect:convergence} above. The bottom of the cube ($z=0$) is fixed while $2/5$ of the lid ($z=1$)
is prescribed an initial displacement of $\bfu_0=(0,0,0.20)$ and there are no external forces. The initial deformation in the whole domain is the steady state solution illustrated in Figure~\ref{fig:initial_state}.
Time integration is performed using a timestep $k=0.01$ and the spatial discretization is done with a mesh size $h=1/5$. After a simulation time of $t=0.1$, the prescribed displacement on the part of the lid is released, and the simulation continues until $T=0.5$.
Recorded energies throughout this simulation are presented in Figure~\ref{fig:conservation}.
Note that this conservation law is fulfilled independent of the timestep.

\begin{figure}
\centering
\includegraphics[width=0.40\linewidth]{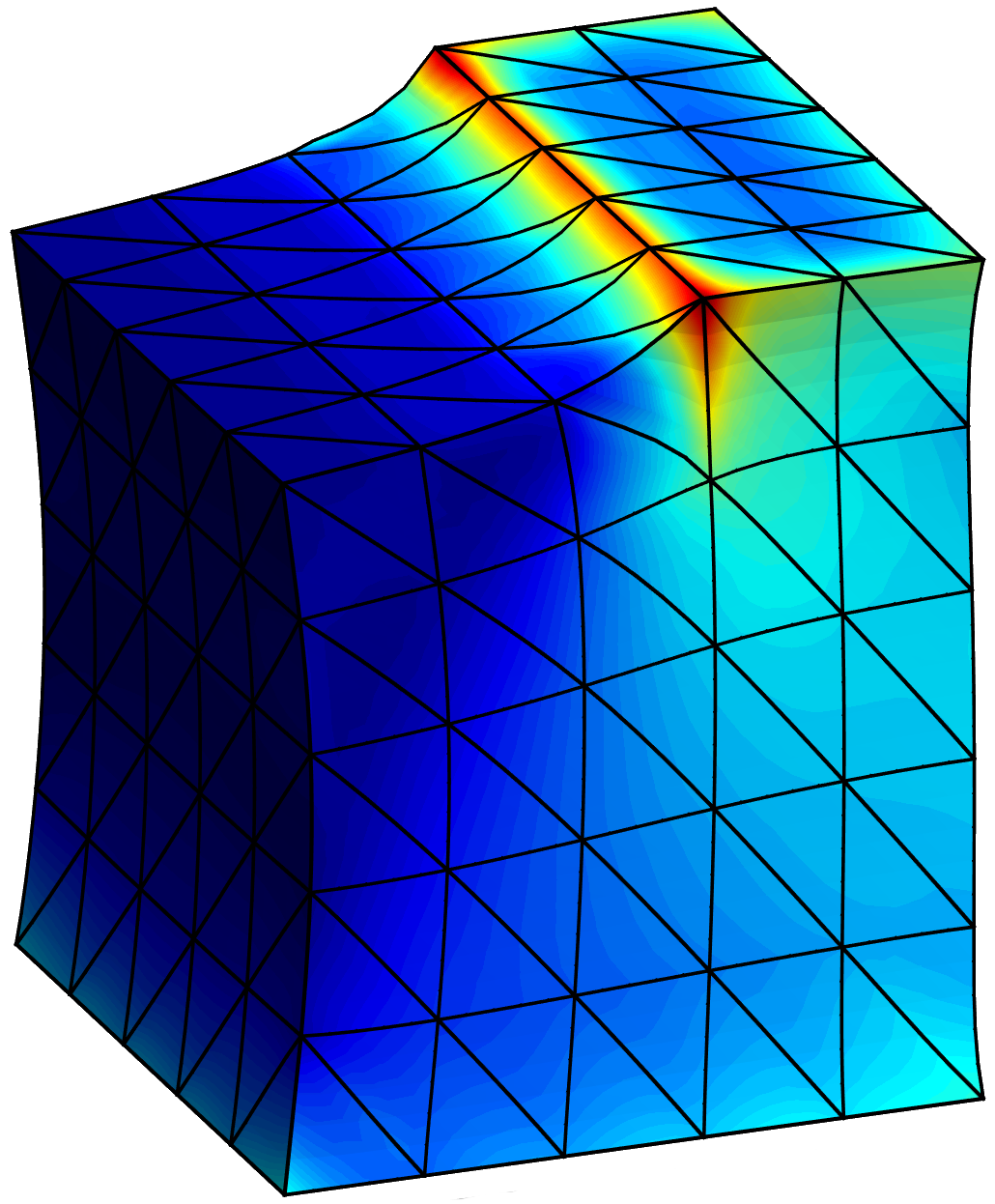}
\caption{Initial state of the deformed unit cube in the conservation law experiment with the color indicating the von Mises stress.
Displacements on part of the lid are initially prescribed to $\bfu=(0,0,0.20)$ and then released.
}
\label{fig:initial_state}
\end{figure}

\begin{figure}
\centering
\includegraphics[width=0.80\linewidth]{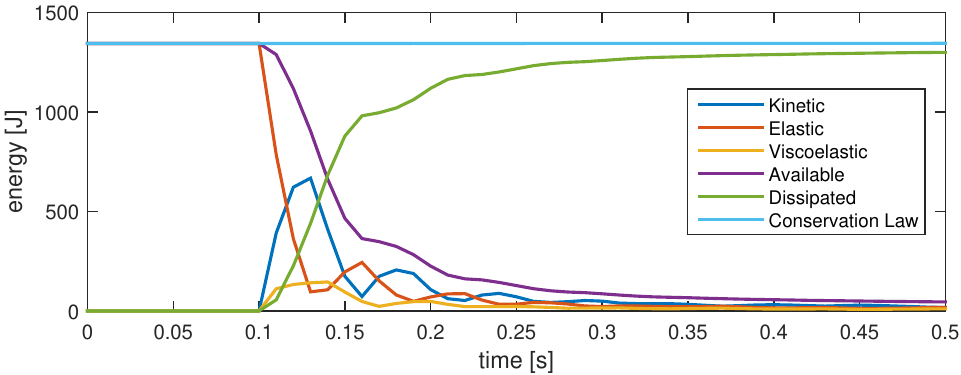}
\caption{
Energies in the discrete conservation law during simulation where an initially deformed cube is released at $t=0.1$. The purple line denoted `Available' is the free energy of the system.
}
\label{fig:conservation}
\end{figure}


\subsection{Numerical Example: A Radial Shaft Seal}
Mechanical seals are used in a wide variety of engineering applications when two surfaces need to be joined to prevent leakage, contamination, contain pressure, etc. Due to many preferred properties, such as mechanical damping, viscoelastic elastomers are used when manufacturing seals.
It is important that a seal holds tight under expected operational conditions, and therefore accurate descriptions of viscoelastic materials are necessary in the design process.

To demonstrate our viscoelastic model, we mimic the following problem: a seal is applied to a vibrating shaft connected to an outer structure, and we wish to investigate the seal holding conditions for a fabricated viscoelastic material under changing shaft vibrations.  

As a simple model of a seal, we use a pipe geometry with inner radius $0.6\times 10^{-2}$, outer radius $1.0\times 10^{-2}$ and length $2.0\times 10^{-2}$, see Figure \ref{fig:seal_domain}. We clamp the outer boundary, i.e., impose homogeneous Dirichlet conditions $\bfu_0 = \mathbf{0}$. To simulate the application of the seal on the shaft, we apply a slip boundary condition to the inner boundary
\begin{align}
\bfu_0 \cdot \bfn &= u_n(\bfx,t) \ , \qquad (\bfI-\bfn\otimes\bfn)\bfsig \bfn = \bfzero \qquad \text{on $\Gamma_{\text{shaft}}$}
\end{align}
where the function $u_n(\bfx,t)$ is defined such that the radial displacement yields a cylindrical shape radially expanded by $1\%$ of the inner radius, moving in a circular motion as illustrated in Figure~\ref{fig:seal_slip_condition}. This boundary condition will mimic the shaft's vibration at different frequencies $\omega$ and the circular motion has an amplitude of $0.1 \%$ of the inner radius. As our fabricated material, we choose the following parameters: elastic modulus $E = 0.5\times 10^6$, Poisson ratio $\nu = 0.39$, and viscoelastic parameters found in Table~\ref{Tab:VE_params}. For spatial discretization, we use standard Lagrange $P_2$ finite elements on tetrahedra. Our model consists of $2880$ elements.

When simulating holding conditions for the seal, we measure the contact pressure $P = \bfn\cdot\bfsig\bfn$, and the seal will hold as long as $P \geq 0$. When negative contact pressure arises, our simulation is no longer valid as the simulation does not include a contact model, and thus, we view any negative contact pressure as an indication of seal failure. In Figure \ref{fig:pressure_shells}, we show the pressure on the inner surface for the frequency $\omega = 10$ at four different times during one vibration cycle.

We investigate the pressure on the inner surface at the end time of three simulations with different frequencies $\omega$; $1$, $7$ and $15$, as seen in Figure \ref{fig:pressure_surfaces}. Non-positive pressures are seen in the cases of $\omega =7$ and $\omega =15$, indicating that the seal might not hold. We then make a sweep over a range of frequencies $ \omega = 10^{-2}$ to $\omega = 10^3$. During the sweep, we measure the pressure at three points along the axial direction, near the endpoint, at $\nicefrac{1}{4}$ and $\nicefrac{1}{2}$ length of the seal. We consider an indication of failure if the pressure at any point goes below zero. We measure during the last cycle for each frequency and plot the minimum and maximum pressures in each point against frequency in Figure \ref{fig:pressure_vs_w_graph}. From this, we see that the minimum pressure goes below zero around $10$ near the endpoint, and up to $10^3$, it remains negative. This supports the indication that the seal might be suitable only for frequencies lower than $10$.

\begin{figure}
\centering
\includegraphics[width=0.5\linewidth]{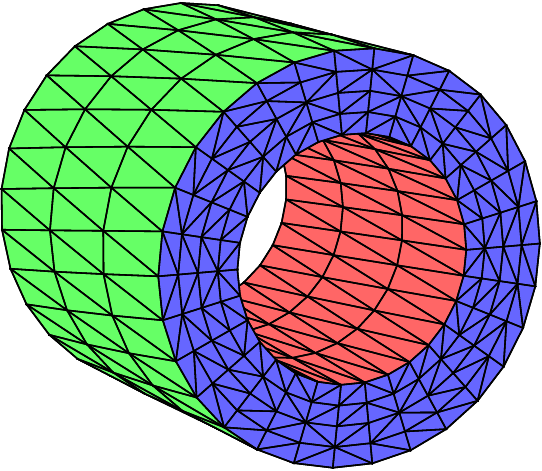}
\caption{The computational domain of the seal. The clamped surface is green, the red surface is subject to a slip boundary condition and a radial displacement, and the blue surfaces are free. }
\label{fig:seal_domain}
\end{figure}

\begin{figure}
\centering
\includegraphics[width=0.95\linewidth]{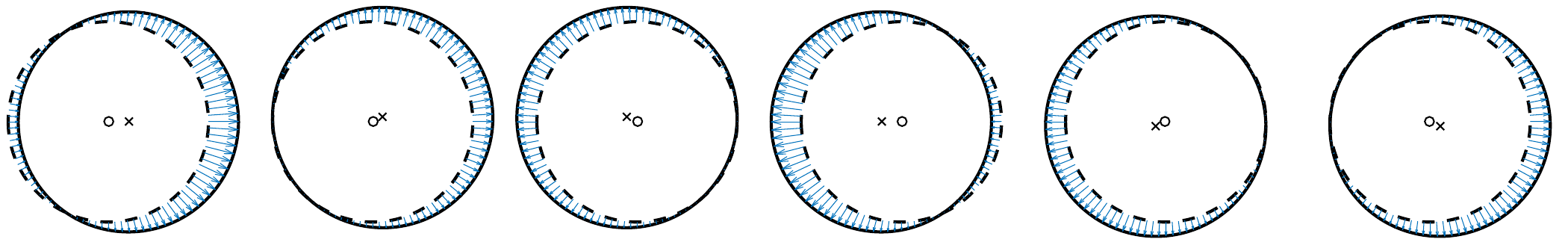}
\caption{Illustration of the slip boundary condition on the inner surface of the seal. The dashed circle illustrates the surface of the material at rest centered around the small circle.
The blue vectors illustrate a prescribed displacement in the surface's normal direction, such that the deformed surface - the solid line - is a circle with a larger radius orbiting the center along an elliptical path.
}
\label{fig:seal_slip_condition}
\end{figure}

\begin{figure}
\centering
\includegraphics[width=.75\linewidth]{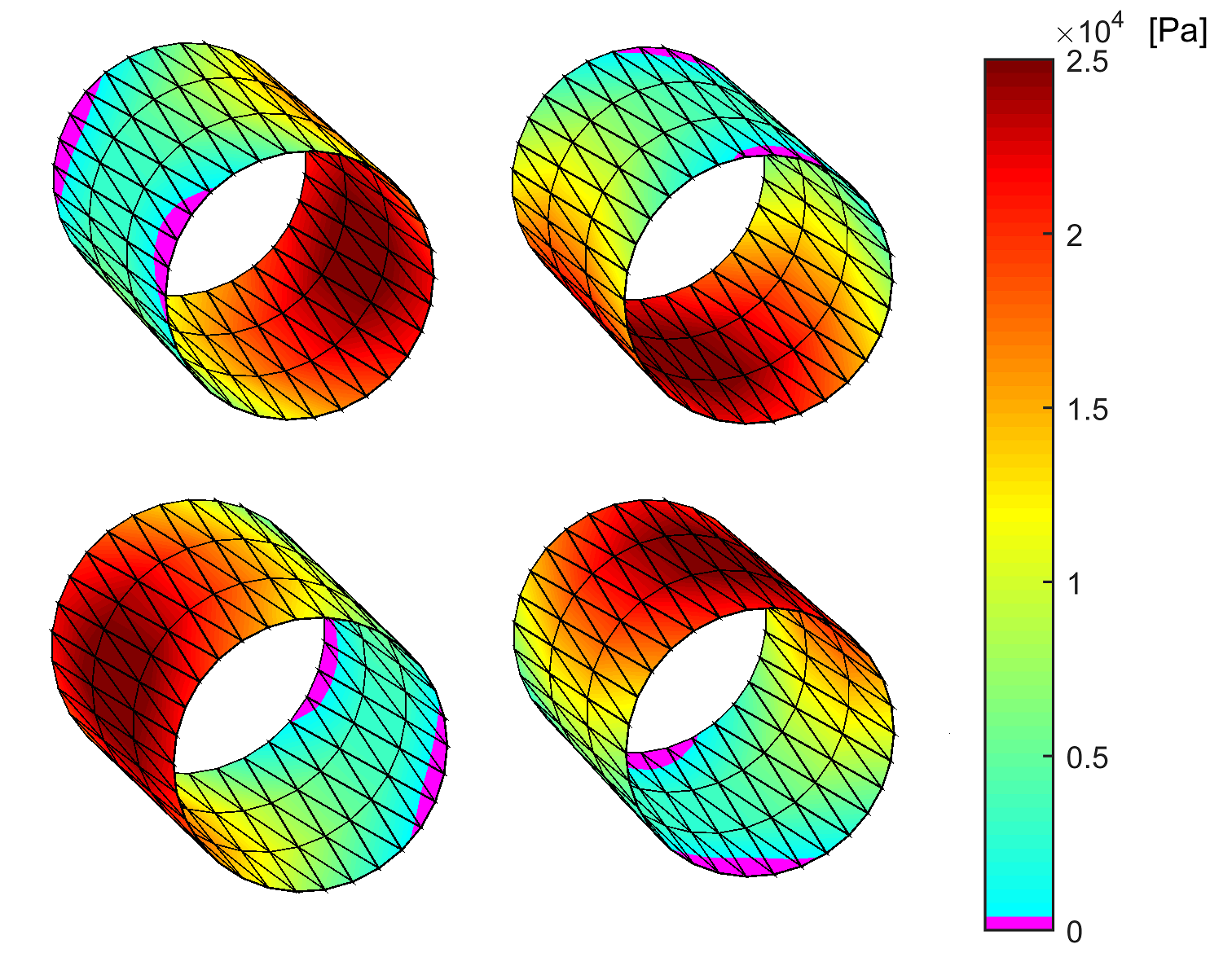}
\caption{Pressure on the inner surface. From top left to bottom right, four different states during the last cycle of a run with $\unit[10]{Hz}$. Pressure below zero is colored pink.}
\label{fig:pressure_shells}
\end{figure}

\begin{figure}
\centering
\includegraphics[width=1\linewidth]{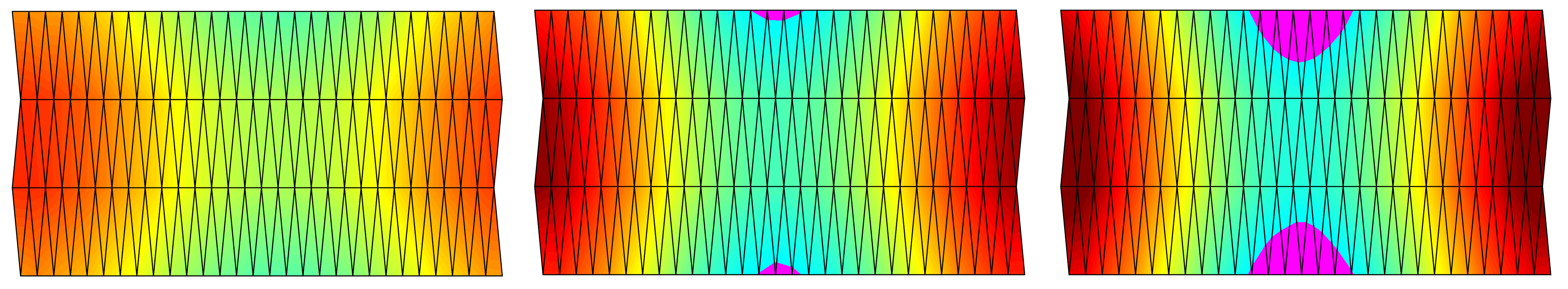}
\caption{The inner shell of the seal is cut and unwrapped as a plane. 
The pressure on the inner surface at the last time-step of three simulations. From left to right for the frequencies $\unit[1]{Hz}$, $\unit[7]{Hz}$, $\unit[15]{Hz}$. Pink denotes negative pressure. }
\label{fig:pressure_surfaces}
\end{figure}

\begin{figure}
\centering
\includegraphics[width=1\linewidth]{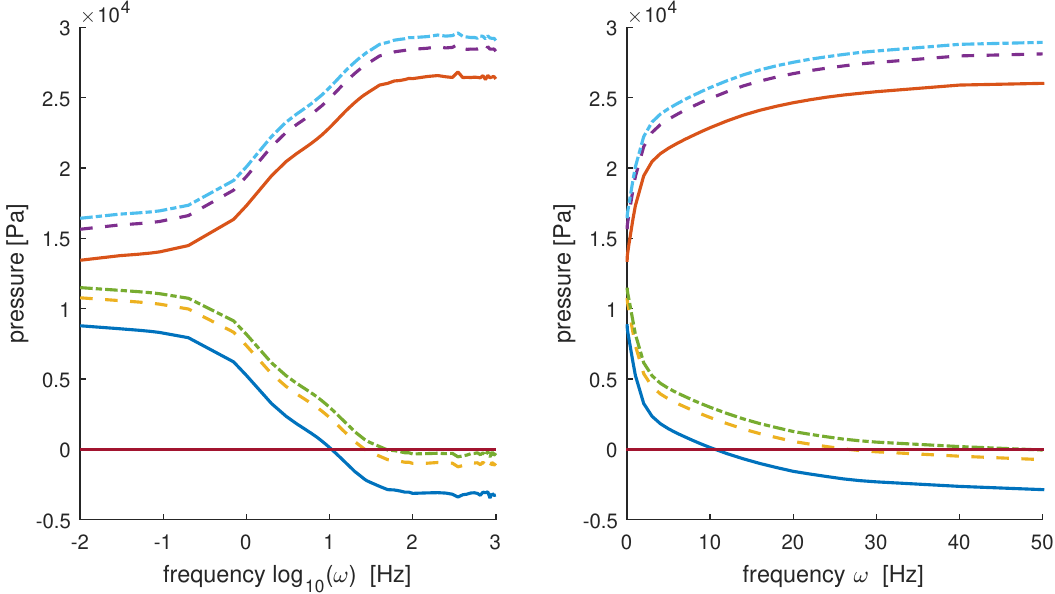}
\caption{Maximum and minimum contact pressure measured at an endpoint (solid), at a quarter of the seal length (dashed), and midpoint (dot-dashed) at different frequencies of the off-center motion of the inner boundary. The left plot shows pressure vs. logarithm frequency spanning $\unit[0.01 - 1000]{Hz}$, and the right plot shows pressure vs. frequency zoomed in between $\unit[0.01 - 50]{Hz}$.
}
\label{fig:pressure_vs_w_graph}
\end{figure}

\begin{table}
\centering
\caption{Viscoelastic parameters for the radial shaft seal problem.}
\begin{tabular}{lccccc}
\toprule
$m$  & 1 & 2 & 3 & 4  & 5 \\
\midrule
$\tau_m$  [\unit{s}] &  $1.0\cdot 10^{-2}$ & $1.0\cdot 10^{-1}$ & $1.0\cdot 10^{0}$ & $1.0\cdot 10^{1}$ & $1.0\cdot 10^{2}$    \\
$\kappa_m$  [\unit{Pa}] &  $3.5 \cdot 10^6$ & $4.0 \cdot 10^6$ & $2.5 \cdot 10^5$ & $2.5 \cdot 10^5$ & $5.0\cdot 10^5$ \\
\bottomrule
\end{tabular}

\label{Tab:VE_params}
\end{table}


\section{Conclusions} \label{sec:conclusions}
In this work we have presented and analyzed a space-time finite element method for the dynamic simulation of linear viscoelastic materials modeled using the generalized Maxwell model. A priori error space-time estimates in both energy norm and in $L^2$ norm were proven. We see the following benefits of our approach:
\begin{itemize}
\item The basis for showing our space-time error estimates is an error representation formula derived using a discrete dual problem. By utilizing different stabilities of the dual solution, this formula yields end-time estimates in different norms. Since inequalities using Grönvall's lemma are avoided, these estimates only linearly depend on the end time $T$.
This framework should readily adapt also to other discretizations in space, for instance dG methods suitable for incompressible or nearly incompressible materials \cite{MR1886000}, which polymers often are.
 
\item In contrast to previous works, our analysis does not use the simplifying assumption of formulating the viscoelastic components of the stress in terms of the same differential operator in space as the elastic stress. Rather, the viscoelastic stress is here more realistically modeled using the deviatoric strain as the differential operator. This choice necessitates a more involved analysis, in particular the bounds proven in Appendix~\ref{sec:appendix}.

\item By eliminating the viscoelastic variables, see Section~\ref{sec:elimination}, the final system is reduced to having only displacements and velocities as unknowns. This yields a very efficient method with essentially the same computational complexity as the dynamic simulation of purely elastic materials.
\end{itemize}
A natural continuation of this work is the extension of the analysis to higher-order approximations in time.

\bigskip
\paragraph{Acknowledgement.} This research was supported in part by the Swedish Research
Council Grants Nos.\  2017-03911,  2021-04925;  and the Swedish
Research Programme Essence.

\bibliographystyle{habbrv}
\footnotesize{
\bibliography{ViscoRef}
}

\bigskip
\noindent
\footnotesize {\bf Authors' addresses:}

\smallskip
\noindent
Martin Björklund  \quad \hfill \addressumushort\\
{\tt martin.bjorklund@math.umu.se}

\smallskip
\noindent
Karl Larsson, \quad \hfill \addressumushort\\
{\tt karl.larsson@umu.se}

\smallskip
\noindent
Mats G. Larson,  \quad \hfill \addressumushort\\
{\tt mats.larson@umu.se}

\clearpage
\appendix
\normalsize
\section{\boldmath Technical Result for the $L^2(\Omega)$ Estimate}
\label{sec:appendix}

For notational simplicity we here assume that we only have one viscoelastic field, i.e., one Maxwell arm in the viscoelastic model. Then, the term \eqref{eq:fbjbbb} that we cannot handle using the standard arguments of our framework takes the form
\begin{align} \label{eq:simpler-term}
II_4 + II_5 &=
\int_0^T a_{VE}( (I-R_E)\bfu_1, \cproj \bfPhi_{VE}) \, dt
\end{align}
The arguments below are easily generalized to the case of multiple Maxwell arms. The main complicating factor in \eqref{eq:simpler-term} is that the Ritz projection $R_E$ does not match the operator in the bilinear form $a_{VE}(\cdot,\cdot)$. If the viscoelastic stresses would instead be expressed in terms of the elastic stress operator, the Ritz projection and the bilinear form would match and hence \eqref{eq:simpler-term} would simply vanish.
Instead, we employ the estimate
\begin{align} \label{eq:technical-result}
\boxed{
|II_4 + II_5|
\lesssim T
(h^{s} + k h^{s-1}) \| \bfu_1 \|_{L^\infty(H^s(\Omega))} \|\mcL_{h,E} \bfPsi_0 \|_{L^2(\Omega)}
}
\end{align}
which we prove below. The general idea is to utilize that we have the stability for 
\begin{align} \label{eq:stability-we-want-to-use}
\|\mcL_{h,E} \bfPhi_0 + \mcL_{h,VE} \bfPhi_{VE}\|_{L^\infty(L^2(\Omega))}
\end{align}
but note that we do not have separate control over $\mcL_{h,E} \bfPhi_0$ or $\mcL_{h,VE} \bfPhi_{VE}$.

\subsection{Modified Operator}
There is a close relationship between $\bfPhi_{VE}$ and $\bfPhi_{0}$, which is enclosed in the discrete dual problem \eqref{eq:discdualprob}.
In this section we use this relationship to construct a modified differential operator depending on a single field $\bfPhi_{VE}$, such that the leading order terms have a structure corresponding to the stability we want to utilize \eqref{eq:stability-we-want-to-use}.

\paragraph{\boldmath Expressing $\bfPhi_{VE}$ in Terms of $\bfPhi_{0}$.}
The relationship between $\bfPhi_{VE}$ and $\bfPhi_{0}$ over a time interval $I_n$ can be formulated as the ODE
\begin{align}
P_0 \left(-\dot{\bfPhi}_{VE}+\dot{\bfPhi}_{0}+\frac{1}{\tau}\bfPhi_{VE}\right)=0
\end{align}
which has the explicit solution
\begin{align}
\bfPhi_{VE}(t_{n-1}) = 
\underbrace{\frac{2}{2+k/\tau}}_{=:\alpha}\bigl(\bfPhi_0(t_{n-1})-\bfPhi_0(t_{n})\bigr)
+
\underbrace{\frac{2-k/\tau}{2+k/\tau}}_{=:\beta} \bfPhi_{VE}(t_{n})
\end{align}
and we note that $0 < \alpha<1$ and $|\beta| < 1$.
Given that $\bfPhi_{VE}(T)=0$, this recursive relationship can be expanded entirely in $\bfPhi_0$ as
\begin{align}
\bfPhi_{VE}(t_{n-1}) &= \sum_{i=n}^{N} \alpha \beta^{i-n}
\bigl(\bfPhi_0(t_{i-1})-\bfPhi_0(t_{i})\bigr)
\\&=
\alpha
\bfPhi_0(t_{n-1})
-\alpha \beta^{N-n} \bfPhi_0(T)
+
\alpha
\sum_{i=n}^{N-1} \beta^{i-n} (\beta-1)\bfPhi_0(t_{i})
\\&=
\alpha
\bfPhi_0(t_{n-1})
-\alpha \beta^{N-n} \bfPhi_0(T)
-
\frac{\alpha^2 k}{\tau}
\sum_{i=n}^{N-1} \beta^{i-n} \bfPhi_0(t_{i})
\label{eq:Phi0-expansion}
\end{align}
where we used $\beta - 1 = \frac{2-k/\tau}{2+k/\tau} - \frac{2+k/\tau}{2+k/\tau} = -\frac{2k/\tau}{2+k/\tau} = -\alpha k/\tau$. Note that the coefficients in the sum in \eqref{eq:Phi0-expansion} are exponentially decaying, since $|\beta|<1$.

\paragraph{Differential Operator.}
To get a differential operator with the correct structure,
we to \eqref{eq:simpler-term} add a term that is zero by the definition of the Ritz projection $R_E$, such that
\begin{align}
II_4 + II_5 &=
\int_0^T \underbrace{a_E((I-R_E) \bfu_1, \alpha^{-1}\cproj\bfPhi_{VE})}_{=0}
+
a_{VE}( (I-R_E)\bfu_1, \cproj \bfPhi_{VE}) \, dt
\\ &=
\int_0^T \bigl(\bfeps(\cproj(I-R_E) \bfu_1), \underbrace{\alpha^{-1}\bfsig_{E}(\bfPhi_{VE})
+
\bfsig_{VE}(\bfPhi_{VE})}_{=:\bfsig_\bigstar(\bfPhi_{VE})} \bigr) \, dt
\label{eq:modified-term}
\end{align}
Since $0 < \alpha < 1$, the bilinear form $(\bfsig_\bigstar(\bfv),\bfeps(\bfw))$ is coercive with energy norm $\tn \bfv \tn_\bigstar^2 := (\bfsig_\bigstar(\bfv),\bfeps(\bfv))$.

\subsection{Discrete Elliptic Regularity}
In this section, we will prove the bound
\begin{align} \label{eq:disc-elliptic-reg}
| ( \bfeps(\bfw) , \bfsig_\bigstar(\bfPhi_{VE}) ) |
&\lesssim
\left( \| \bfw \|_{L^2(\Omega)}^2 + h^2 \| \bfw\otimes \nabla \|_{L^2(\Omega)}^2 \right)^{1/2}
\\&\qquad \nonumber
\cdot\| \mcL_{h,E} (\alpha^{-1}\bfPhi_{VE}) + \mcL_{h,VE} \bfPhi_{VE} \|_{L^2(\Omega)}
\end{align}
for $\bfw\in \V$.

\paragraph{Continuous Field.}
We here introduce a continuous field $\tilde{\bfPhi}_{VE}$ such that the discrete field $\bfPhi_{VE}$ is the Ritz projection of $\tilde{\bfPhi}_{VE}$.
By the definitions of the discrete differential operators \eqref{eq:discrete-operator-E} and \eqref{eq:discrete-operator-VE}, we see that $\bfPhi_{VE}\in\Vh$ satisfies
\begin{align}\label{eq:star-disc}
(\bfsig_\bigstar(\bfPhi_{VE}),\bfeps(\bfv))
= 
(\mcL_{h,E} (\alpha^{-1}\bfPhi_{VE}) + \mcL_{h,VE} \bfPhi_{VE} , \bfv )
\qquad\text{for all $\bfv\in\Vh$}
\end{align}
We introduce the corresponding continuous problem; let $\tilde{\bfPhi}_{VE}$ be the solution to
\begin{alignat}{2}\label{eq:star-strong}
\bfsig_\bigstar(\tilde{\bfPhi}_{VE})\cdot\nabla &= \mcL_{h,E} (\alpha^{-1}\bfPhi_{VE}) + \mcL_{h,VE} \bfPhi_{VE} &\qquad&\text{in $\Omega$}
\\
\label{eq:star-neumann}
\bfsig_\bigstar(\tilde{\bfPhi}_{VE})\bfn &=\bfzero &\qquad&\text{on $\Gamma_N$}
\\
\tilde{\bfPhi}_{VE} &= \bfzero &\qquad&\text{on $\Gamma_D$}
\end{alignat}
which in weak form reads; find $\tilde{\bfPhi}_{VE}\in\V$ such that
\begin{align}\label{eq:star-cont}
(\bfsig_\bigstar(\tilde{\bfPhi}_{VE}),\bfeps(\bfv))
= 
(\mcL_{h,E} (\alpha^{-1}\bfPhi_{VE}) + \mcL_{h,VE} \bfPhi_{VE} , \bfv )
\qquad\text{for all $\bfv\in\V$}
\end{align}
For this continuous problem, we have the elliptic regularity
\begin{align} \label{eq:star-regularity}
| \tilde{\bfPhi}_{VE} |_{H^2(\Omega)} \lesssim \| \mcL_{h,E} (\alpha^{-1}\bfPhi_{VE}) + \mcL_{h,VE} \bfPhi_{VE} \|_{L^2(\Omega)}
\end{align}
Comparing \eqref{eq:star-disc} and \eqref{eq:star-cont}, we realize that $\bfPhi_{VE}$ is the Ritz projection of $\tilde{\bfPhi}_{VE}$, with energy norm error bound
\begin{align} \label{eq:star-ritz-est}
\tn \tilde{\bfPhi}_{VE} - \bfPhi_{VE} \tn_\bigstar
\lesssim h | \tilde{\bfPhi}_{VE} |_{H^2(\Omega)}
\lesssim h \| \mcL_{h,E} (\alpha^{-1}\bfPhi_{VE}) + \mcL_{h,VE} \bfPhi_{VE} \|_{L^2(\Omega)}
\end{align}

\paragraph{Integration by Parts.}
For compactness in the calculations below, we here let
\begin{align}
\bfsig_h
&:=
\bfsig_\bigstar(\bfPhi_{VE})
, \qquad
\tilde{\bfsig}
:=
\bfsig_\bigstar(\tilde{\bfPhi}_{VE})
, \qquad
\pi_h\tilde{\bfsig}
:=
\bfsig_\bigstar(R_{VE}\tilde{\bfPhi}_{VE}),
\end{align}
Integrating by parts elementwise and applying the Cauchy-Schwarz inequality, give us
\begin{align}
| ( \bfeps(\bfw) , \bfsig_h ) |
& =
\Bigl|
\sum_{K\in\mcK_h}
\bigl(\bfw, \bfsig_h \bfn \bigr)_{\partial K}
- \bigl(\bfw, \bfsig_h \cdot\nabla \bigr)_K
\Bigr|
\\& =
\Bigl|
\bigl(\bfw, \llbracket \bfsig_h \bfn_F \rrbracket \bigr)_{\mcF_h}
-
\bigl(\bfw, \bfsig_h\cdot\nabla \bigr)_{\mcK_h}
\Bigr|
\\& \leq
\Bigl|
\bigl(\bfw, \llbracket \bfsig_h \bfn_F \rrbracket \bigr)_{\mcF_h}
\Bigr|
+
\Bigl|
\bigl(\bfw, \bfsig_h \cdot\nabla \bigr)_{\mcK_h}
\Bigr|
\\& \leq
h^{1/2}\| \bfw \|_{L^2(\mcF_h)}
h^{-1/2} \bigl\| \llbracket \bfsig_h \bfn_F \rrbracket \bigr\|_{L^2(\mcF_h)}
+ \| \bfw \|_{L^2(\Omega)}
\bigl\| \bfsig_h \cdot\nabla \bigr\|_{L^2(\mcK_h)}
\\& \leq
\left( h \| \bfw \|_{L^2(\mcF_h)}^2 + \| \bfw \|_{L^2(\Omega)}^2 \right)^{1/2}
\\ & \ \quad\cdot \nonumber
\Bigl( h^{-1}\bigl\| \llbracket \bfsig_h \bfn_F \rrbracket \bigr\|_{L^2(\mcF_h)}^2 + \bigl\| \bfsig_h\cdot\nabla \bigr\|_{L^2(\mcK_h)}^2 \Bigr)^{1/2}
\\& \lesssim
\left( \| \bfw \|_{L^2(\Omega)}^2 + h^2 \| \bfw\otimes \nabla \|_{L^2(\Omega)}^2 \right)^{1/2}
\\ & \ \quad\cdot \nonumber
\Bigl(
\underbrace{h^{-1}\bigl\| \llbracket \bfsig_h \bfn_F \rrbracket \bigr\|_{L^2(\mcF_h)}^2}_{I_\bigstar}
+
\underbrace{\bigl\| \bfsig_h\cdot\nabla \bigr\|_{L^2(\mcK_h)}^2}_{II_\bigstar}
\Bigr)^{1/2}
\end{align}
where we in the final inequality applied the standard elementwise trace inequality
\begin{align} \label{eq:std-trace}
\| \bfv \|_{L^2(\partial K)}^2 \lesssim h^{-1}\| \bfv \|_{L^2(K)}^2 + h\| \bfv\otimes \nabla \|_{L^2(K)}^2
\end{align}
Here, $\mathcal{F}_h$ denotes the set of all interior edges in $\mathcal{K}_h$, as well as all edges on the Neumann boundary, where we define the jump $\llbracket \bfsig_h \bfn_F \rrbracket|_{\Gamma_N} = \bfsig_h \bfn|_{\Gamma_N}$.

\paragraph{\boldmath Term $I_\bigstar$.}
Thanks to the regularity of $\tilde{\bfPhi}_{VE}$ and that $\tilde{\bfsig}\bfn=\bfzero$ on $\Gamma_N$, we can subtract $\tilde{\bfsig}$ inside the jump $\llbracket \cdot \rrbracket$ without affecting the value of the term. Thereafter applying the trace inequality \eqref{eq:std-trace}, we get
\begin{align}
h^{-1}\bigl\| \llbracket \bfsig_h \bfn_F \rrbracket \bigr\|_{L^2(\mcF_h)}^2
&=
h^{-1}\bigl\| \llbracket (\bfsig_h -\tilde{\bfsig}) \bfn_F \rrbracket \bigr\|_{L^2(\mcF_h)}^2
\\&\lesssim
h^{-2}
\| \bfsig_h -\tilde{\bfsig} \|_{L^2(\Omega)}^2
+
\| (\bfsig_h -\tilde{\bfsig})\otimes\nabla \|_{L^2(\mathcal{K}_h)}^2
\label{eq:skjvdb}
\end{align}

For the first term in \eqref{eq:skjvdb} we utilize the error estimate \eqref{eq:star-ritz-est}, giving us
\begin{align} \label{eq:sdivbsiud}
h^{-2}
\| \bfsig_h -\tilde{\bfsig} \|_{L^2(\Omega)}^2
&\sim h^{-1}
\tn \tilde{\bfPhi}_{VE} - \bfPhi_{VE} \tn_\bigstar^2
\lesssim
\| \mcL_{h,E} (\alpha^{-1}\bfPhi_{VE}) + \mcL_{h,VE} \bfPhi_{VE} \|_{L^2(\Omega)}^2
\end{align}

For the second term in \eqref{eq:skjvdb}, we add and subtract the interpolant $\pi_h\tilde{\bfsig}$, use an inverse inequality on the discrete term, whereafter we add and subtract $\tilde\bfsig$ in the lower order term, yielding
\begin{align}
\label{eq:odfbjdfb}
\| (\bfsig_h -\tilde{\bfsig})\otimes\nabla \|_{L^2(\mathcal{K}_h)}^2
&\lesssim
\| (\bfsig_h - \pi_h \tilde{\bfsig})\otimes\nabla \|_{L^2(\mathcal{K}_h)}^2
+
\| (\tilde{\bfsig}-\pi_h \tilde{\bfsig})\otimes\nabla \|_{L^2(\mathcal{K}_h)}^2
\\
&\lesssim
h^{-2}\| \bfsig_h - \pi_h \tilde{\bfsig} \|_{L^2(\Omega)}^2
+
\| (\tilde{\bfsig}-\pi_h \tilde{\bfsig})\otimes\nabla \|_{L^2(\mathcal{K}_h)}^2
\\
&\lesssim \label{eq:bjdsffd}
h^{-2}\| \bfsig_h - \tilde{\bfsig} \|_{L^2(\Omega)}^2
+
h^{-2}\| \tilde{\bfsig} - \pi_h \tilde{\bfsig} \|_{L^2(\mathcal{K}_h)}^2
\\&\quad\nonumber
+
\| (\tilde{\bfsig}-\pi_h \tilde{\bfsig})\otimes\nabla \|_{L^2(\mathcal{K}_h)}^2
\\&
\lesssim
\| \mcL_{h,E} (\alpha^{-1}\bfPhi_{VE}) + \mcL_{h,VE} \bfPhi_{VE} \|_{L^2(\Omega)}^2
\end{align}
For the final inequality, we estimate the first term in \eqref{eq:bjdsffd} as \eqref{eq:sdivbsiud}, and the remaining two terms using interpolation estimates.

\paragraph{\boldmath Term $II_\bigstar$.}
Adding and subtracting the interpolant $\pi_h\tilde{\bfsig}$ as well as $\tilde{\bfsig}$ give
\begin{align}
\bigl\| \bfsig_h\cdot\nabla \bigr\|_{L^2(\mcK_h)}
&\lesssim
\bigl\| (\bfsig_h - \pi_h\tilde{\bfsig} ) \cdot \nabla \bigr\|_{L^2(\mcK_h)}^2
+
\bigl\| ( \pi_h \tilde{\bfsig} - \tilde{\bfsig}) \cdot \nabla \bigr\|_{L^2(\mcK_h)}^2
+
\bigl\| \tilde{\bfsig} \cdot \nabla \bigr\|_{L^2(\mcK_h)}^2
\\
&\lesssim \label{eq:jbvddbsdl}
\bigl\| (\bfsig_h - \pi_h\tilde{\bfsig} ) \otimes \nabla \bigr\|_{L^2(\mcK_h)}^2
+
\bigl\| (\tilde{\bfsig} - \pi_h \tilde{\bfsig}) \otimes \nabla \bigr\|_{L^2(\mcK_h)}^2
+
\bigl\| \tilde{\bfsig} \cdot \nabla \bigr\|_{L^2(\mcK_h)}^2
\\&
\lesssim
\| \mcL_{h,E} (\alpha^{-1}\bfPhi_{VE}) + \mcL_{h,VE} \bfPhi_{VE} \|_{L^2(\Omega)}^2
\end{align}
In the right hand side of \eqref{eq:jbvddbsdl}, we note that the first two terms are the same as the right hand side of \eqref{eq:odfbjdfb}, and for the final term we use the strong form of the continuou problem \eqref{eq:star-strong}.
This completes the proof of \eqref{eq:disc-elliptic-reg}.

\subsection{Stability Bound}
We here show that we have stability of the term
\begin{align}
\| \mcL_{h,E} (\alpha^{-1}\bfPhi_{VE}) + \mcL_{h,VE} \bfPhi_{VE} \|_{L^\infty(L^2(\Omega))}
\end{align}
By calculations analogous to \eqref{eq:time-max-start}--\eqref{eq:time-max-end}, we note that $\| \mcL_{h,E} (\alpha^{-1}\bfPhi_{VE}) + \mcL_{h,VE} \bfPhi_{VE} \|_{L^2(\Omega)}$ must attain its maximum value at a node in time. Hence, we below assume $t=t_{n-1}$ and that this is the point in time where the maximum value is attained.
Adding and subtracting terms and using the triangle inequality gives
\begin{align}
\| \mcL_{h,E} (\alpha^{-1}\bfPhi_{VE}) + \mcL_{h,VE} \bfPhi_{VE} \|_{L^2(\Omega)}
&\leq
\| \mcL_{h,E} \bfPhi_0 + \mcL_{h,VE} \bfPhi_{VE} \|_{L^2(\Omega)}
\\&\qquad\nonumber
+
\| \mcL_{h,E}(\bfPhi_{0} - \alpha^{-1}\bfPhi_{VE}) \|_{L^2(\Omega)}
\end{align}
where we have stability for the first term.
In the second term, we replace $\bfPhi_{VE}$ by its expansion in $\bfPhi_{0}$ given by \eqref{eq:Phi0-expansion}, and we get
\begin{align}
&\| \mcL_{h,E}(\bfPhi_{0} - \alpha^{-1}\bfPhi_{VE}) \|_{L^2(\Omega)}
\nonumber
\\
&\qquad\leq
|\beta|^{N-n}\| \mcL_{h,E} \bfPhi_0(T) \|_{L^2(\Omega)}
+
\frac{\alpha}{\tau}
\sum_{i=n}^{N-1} k \beta^{i-n}
\| \mcL_{h,E} \bfPhi_0(t_{i}) \|_{L^2(\Omega)}
\\
&\qquad\leq
\| \mcL_{h,E} \bfPhi_0(T) \|_{L^2(\Omega)}
+
\max_{n\leq i \leq N}  k \| \mcL_{h,E}\bfPhi_0(t_{i}) \|_{L^2(\Omega)}
\frac{\alpha^2}{\tau}
\sum_{i=n}^{N}
\beta^{i-n}
\\
&\qquad\lesssim
\| \mcL_{h,E} \bfPsi_0 \|_{L^2(\Omega)}
+
\max_{n\leq i \leq N}  k \| \mcL_{h,E}\bfPhi_0(t_{i}) \|_{L^2(\Omega)}
\end{align}
where we use $\sum_{i=n}^{N} \beta^{i-n} \leq \frac{1}{1-\beta}$ since $|\beta|<1$.
Repeatedly adding and subtracting suitable terms and using the triangle inequality, yield
\begin{align}
\| \mcL_{h,E} \bfPhi_0 \|_{L^2(\Omega)}
&\leq
\| \mcL_{h,E} \bfPhi_0 + \mcL_{h,VE} \bfPhi_{VE} \|_{L^2(\Omega)}
+
\| \mcL_{h,VE} \bfPhi_{VE} \|_{L^2(\Omega)}
\\
&\lesssim \label{eq:dlfbkfdsbk}
\| \mcL_{h,E} \bfPhi_0 + \mcL_{h,VE} \bfPhi_{VE} \|_{L^2(\Omega)}
\\&\qquad\nonumber
+
\bigl\| \mcL_{h,VE}\bigl(-\bfPhi_{1} + \frac{1}{\tau} \bfPhi_{VE} \bigr) \bigr\|_{L^2(\Omega)}
+
\| \mcL_{h,VE} \bfPhi_{1} \|_{L^2(\Omega)}
\end{align}
where the first term in \eqref{eq:dlfbkfdsbk} is directly bounded by our stability.
For the remaining two terms in \eqref{eq:dlfbkfdsbk} we use that $\| \mcL_{h,VE} \bfv \|_{L^2(\Omega)} \lesssim h^{-1} \tn \bfv \tn_{VE}$ for $\bfv\in\Vh$, which follows from the calculation
\begin{align}
\| \mcL_{h,VE} \bfv \|_{L^2(\Omega)}^2
&= (\mcL_{h,VE} \bfv,\mcL_{h,VE} \bfv)
\\&= (\bfsig_{VE}(\bfv),\bfeps(\mcL_{h,VE} \bfv))
\\&= (\bfsig_{VE}(\bfv),  (\mcL_{h,VE} \bfv) \otimes\nabla)
\\&\leq \kappa^{1/2} \tn \bfv \tn_{VE}
\| (\mcL_{h,VE} \bfv) \otimes\nabla\|_{L^2(\Omega)}
\\&\lesssim h^{-1} \tn \bfv \tn_{VE}
\|\mcL_{h,VE} \bfv \|_{L^2(\Omega)}
\end{align}
where we used an inverse inequality in the final inequality.
Hence, we arrive at the bound
\begin{align}
\max_{n\leq i \leq N}
k \| \mcL_{h,E} \bfPhi_0 \|_{L^2(\Omega)}
&\lesssim
k \| \mcL_{h,E} \bfPhi_0 + \mcL_{h,VE} \bfPhi_{VE} \|_{L^\infty(L^2(\Omega))}
\\&\qquad\nonumber
+
k h^{-1} \bigl\| -\bfPhi_{1} + \frac{1}{\tau} \bfPhi_{VE} \bigr\|_{L^\infty(VE)}
+
k h^{-1} \| \bfPhi_{1} \|_{L^\infty(E)}
\end{align}
where we for the third term used $\tn \bfv \tn_{VE} \lesssim \tn \bfv \tn_{E}$.
In summary, we have the stability
\begin{align} \label{eq:star-final-stability}
\| \mcL_{h,E} (\alpha^{-1}\bfPhi_{VE}) + \mcL_{h,VE} \bfPhi_{VE} \|_{L^\infty(L^2(\Omega))}
\lesssim (1 + k h^{-1})\| \mcL_{h,E} \bfPsi_0 \|_{L^2(\Omega)}
\end{align}

\subsection{Final Estimate}

Let $\bfw=\cproj(I-R_E) \bfu_1$. Continuing from \eqref{eq:modified-term}, applying the discrete elliptic regularity \eqref{eq:disc-elliptic-reg} followed by an interpolation estimate and the stability bound \eqref{eq:star-final-stability}, we have
\begin{align}
| II_4 + II_5 | &=
\biggl|
\int_0^T \bigl(\bfeps(\bfw), \bfsig_\bigstar(\bfPhi_{VE}) \bigr) \, dt
\biggr|
\\&\lesssim
\int_0^T 
\left( \| \bfw \|_{L^2(\Omega)}^2 + h^2 \| \bfw\otimes \nabla \|_{L^2(\Omega)} \right)^{1/2}
\\&\qquad \nonumber
\cdot\| \mcL_{h,E} (\alpha^{-1}\bfPhi_{VE}) + \mcL_{h,VE} \bfPhi_{VE} \|_{L^2(\Omega)}
 \, dt
\\&\lesssim
\int_0^T 
h^{s} \| \bfu_1 \|_{H^{s}(\Omega)}
(1 + k h^{-1})\| \mcL_{h,E} \bfPsi_0 \|_{L^2(\Omega)}
 \, dt
\\&\leq
T (h^{s} + k h^{s-1}) \| \bfu_1 \|_{L^\infty(H^s(\Omega))} \|\mcL_{h,E} \bfPsi_0 \|_{L^2(\Omega)}
\end{align}
which concludes the proof of \eqref{eq:technical-result}.

\end{document}